\documentclass[a4paper]{amsart}
\usepackage[utf8x]{inputenc}
\usepackage{graphicx,color,easybmat}
\usepackage{amssymb,amsmath}

\newtheorem{MainTheorem}{Theorem}

\newtheorem*{conjecture*}{Conjecture}
\newtheorem{MainCorollary}[MainTheorem]{Corollary}

\newtheorem{MainProposition}[MainTheorem]{Proposition}

\newtheorem{theorem}{Theorem}[section]
\newtheorem{proposition}[theorem]{Proposition}
\newtheorem{lemma}[theorem]{Lemma}

\theoremstyle{definition}
\newtheorem{remark}[theorem]{Remark}
\newtheorem{defi}[theorem]{Definition}

\newenvironment{case}[2][Case]{%
  \trivlist \item[\hskip\labelsep{\itshape #1 #2.}]\begin{em}}{
  \end{em}\endtrivlist}

\newcommand{\evalat}[1]{\bigr\rvert_{#1}}
\usepackage{enumerate}
\makeatletter
\renewcommand\theenumi{\@alph\c@enumi}
\renewcommand\theenumii{\@alph\c@enumii}
\renewcommand\theenumiii{\@alph\c@enumiii}
\renewcommand\theenumiv{\@alph\c@enumiv}

\def\@map#1#2[#3]{\mbox{$#1 \colon #2 \longrightarrow #3$}}
\def\map#1#2{\@ifnextchar [{\@map{#1}{#2}}{\@map{#1}{#2}[#2]}}
\def\@chull#1[#2]{\ensuremath{\langle #1 \rangle_{_{#2}}}}
\def\chull#1{\@ifnextchar [{\@chull{#1}}{\ensuremath{\langle #1 \rangle}}}
\makeatother
\DeclareMathOperator{\Int}{Int}
\DeclareMathOperator{\Cl}{Cl}
\DeclareMathOperator{\En}{En}

\begin{document}
\title[Zero entropy cycles on trees: from Topology to Combinatorics]{Zero entropy cycles on trees: from Topology to Combinatorics and an application to star maps}

\author[D. Juher]{David Juher}
\address{Departament d'Inform\`atica, Matem\`atica Aplicada i Estad\'{\i}stica,
Universitat de Gi\-ro\-na, c/ Maria Aur\`elia Capmany 61, 17003 Girona, Spain.
ORCID 0000-0001-5440-1705}
\email{david.juher@udg.edu \textrm{(Corresponding author)}}

\author[F. Ma\~{n}osas]{Francesc Ma\~{n}osas}
\address{Departament de Matem\`atiques, Edifici C, Universitat
Aut\`onoma de Barcelona, 08913 Cerdanyola del Vall\`es, Barcelona, Spain. ORCID 0000-0003-2535-0501}
\email{Francesc.Manosas@uab.cat}

\author[D. Rojas]{David Rojas}
\address{Departament de Matem\`atiques, Edifici C, Universitat
Aut\`onoma de Barcelona, 08913 Cerdanyola del Vall\`es, Barcelona, Spain. ORCID 0000-0001-7247-4705}
\email{David.Rojas@uab.cat}

\thanks{This work has been funded by grant PID2023-146424NB-I00 of Ministerio de Ciencia e Innovaci\'on.}


\begin{abstract}
In this paper we give a fully combinatorial description of the zero entropy periodic patterns on trees.
Unlike previously known characterizations of such patterns, our criterion is independent of any particular
topological realization of the pattern and provides, thus, a practical and fast algorithm to test zero entropy.
As an application, consider a $k$-star $T$ (a tree with $k$ edges attached at a unique branching point of
valence $k$) and the set $\mathcal{F}_{n,k}$ of all continuous maps $\map{f}{T}$ having a periodic orbit of period $n$
properly contained in $T$ (each edge of $T$ contains at least one point of the orbit). We find all pairs $(n,k)$
such that $\mathcal{F}_{n,k}$ contains maps of entropy zero, and we describe the patterns of such zero-entropy orbits.
\end{abstract}

\maketitle

\section{Introduction: patterns and canonical models for tree maps}\label{S1}
The notion of \emph{pattern} is central in the field of Combinatorial Dynamics. Consider a class $\mathcal{X}$ of topological
spaces (closed intervals, trees, graphs and compact surfaces are classic examples) and the family $\mathcal{F}_\mathcal{X}$ of
all maps $\{\map{f}{X}:X\in\mathcal{X}\}$ satisfying a given property (continuous maps and homeomorphisms are typical examples).
The iteration of any of such maps gives rise to a discrete dynamical system. Assume now that we have a map $\map{f}{X}$ in
$\mathcal{F}_\mathcal{X}$ which is known to have an invariant set $P$. Broadly speaking, the \emph{pattern of $P$} is the
equivalence class of all maps $\map{g}{Y}$ in $\mathcal{F}_\mathcal{X}$ having an invariant set $Q\subset Y$ that
``behaves like $P$'' in a combinatorial sense.

An example is given by the family $\mathcal{F}_\mathcal{M}$ of surface homeomorphisms. Here the pattern (also called \emph{braid type})
of a cycle $P$ of a map $\map{f}{M}$ from $\mathcal{F}_\mathcal{M}$, where $M$ is a surface, is defined by the isotopy class,
up to conjugacy, of $f\evalat{M\setminus P}$ \cite{bow,mat}.

For simply connected spaces, a definition of pattern of an invariant set $P$ of a map $f$ has to account for:
\begin{enumerate}
\item the relative positions of the points of $P$ inside the space
\item the way these positions are permuted under the action of $f$.
\end{enumerate}

For continuous maps of closed intervals, the points of $P$ are totally ordered and the pattern of $P$ can be simply identified with a
permutation in a natural way. This seminal notion of pattern was formalized and developed in the early 1990s \cite{bald,mn}.

In the last decades, a growing interest has arisen in extending this concept from the interval case to more general
one-dimensional spaces such as graphs \cite{patgraf,AMM} or trees \cite{aglmm,bald2,bern}. In this paper we deal with patterns of
periodic orbits of continuous maps defined on trees. To precise the conditions (a,b) above in this context, next
we introduce some definitions.

A \emph{tree} is a compact uniquely arcwise connected space which is either a point or a union of a finite number of intervals
(an \emph{interval} is any space homeomorphic to $[0,1]\subset\mathbb{R}$). Any continuous map $\map{f}{T}$ from a tree $T$ into itself will be
called a \emph{tree map}. A set $X\subset T$ is said to be \emph{$f$-invariant} if $f(X)\subset X$. For each $x\in T$, we define the
\emph{valence} of $x$ to be the number of connected components of $T\setminus\{x\}$. A point of valence different from 2 will be called
a \emph{vertex} of $T$ and the set of vertices of $T$ will be denoted by $V(T)$. Each point of valence 1 will
be called an \emph{endpoint} of $T$, and the set of such points will be denoted by $\En(T)$. The closure of
a connected component of $T \setminus V(T)$ will be called an \emph{edge of $T$}.

Let $k\in\mathbb{N}$ with $k\ge3$. A tree $T$ will be called a \emph{$k$-star} if $T$ is the union of $k$ edges, which will be
called \emph{branches}, intersecting at a unique vertex of valence $k$, which will be called the \emph{central point}.
A \emph{$1$-star} is a tree consisting of a single point, and a \emph{$2$-star} is an interval.

Given any subset $X$ of a topological space, we will denote by $\Int(X)$ and $\Cl(X)$ the interior and the closure
of $X$, respectively. For a finite set $P$, we will denote its cardinality by $|P|$.

Let $T$ be a tree. Given $X\subset T$ we will define the \emph{convex hull} of $X$, denoted by
$\chull{X}_T$ or simply by $\chull{X}$, as the smallest closed connected subset of $T$ containing $X$. When
$X=\{x,y\}$ we will write $[x,y]$ to denote $\chull{X}$. The notations $(x,y)$, $(x,y]$ and $[x,y)$ will be understood
in the natural way.

A pair $(T,P)$ where $T$ is a tree and $P\subset T$ is a finite subset of $T$ will be called a \emph{pointed tree}.
Two points $x,y$ of $P$ will be called \emph{consecutive} if $(x,y) \cap P = \emptyset$.
Any maximal subset of $P$ consisting only of pairwise con\-sec\-u\-tive points will be called a \emph{discrete component of $(T,P)$}.
See Figures~\ref{exemples} and
\ref{exemples2} for several examples of pointed trees, where each dotted closed curve surrounds a discrete component.

Next we introduce an equivalence relation on pointed trees. Before doing it, let us explain which
features of a pointed tree $(T,P)$ should be preserved inside its class:
\begin{enumerate}
\item[(1)] The number of distinguished points (cardinality of $P$).
\item[(2)] The spatial relations (pairwise relative positions) among the distinguished points, independently of the particular topology of $T$.
\end{enumerate}
While (1) is clear enough, let us precise (2) by showing several examples of pointed trees that should be equivalent
in our definition. See Figure~\ref{exemples}, where the pointed trees $(T_1,P_1)$, $(T_2,P_2)$ and $(T_3,P_3)$ should be equivalent.
See also Figure~\ref{exemples2}, where $(S_1,Q_1)$ and $(S_2,Q_2)$ should be equivalent.

\begin{figure}
\centering
\includegraphics[scale=0.6]{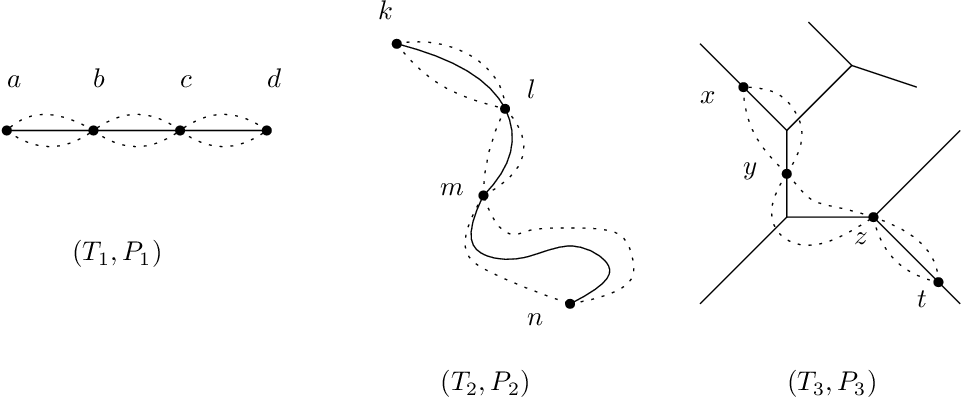}
\caption[fck]{Three pointed trees that should be equivalent. \label{exemples}}
\end{figure}

Looking carefully at the previous examples, the reader may guess that the feature (2) that we want to preserve between two
equivalent pointed trees $(T,P),(S,Q)$ is the following one: there is a bijection $\phi$ between the distinguished sets of points $P$ and $Q$
in such a way that any two points $x,y$ of $P$ are consecutive in $T$ if and only if the corresponding points $\phi(x),\phi(y)$ of $Q$
are consecutive in $S$.

\begin{figure}
\centering
\includegraphics[scale=0.6]{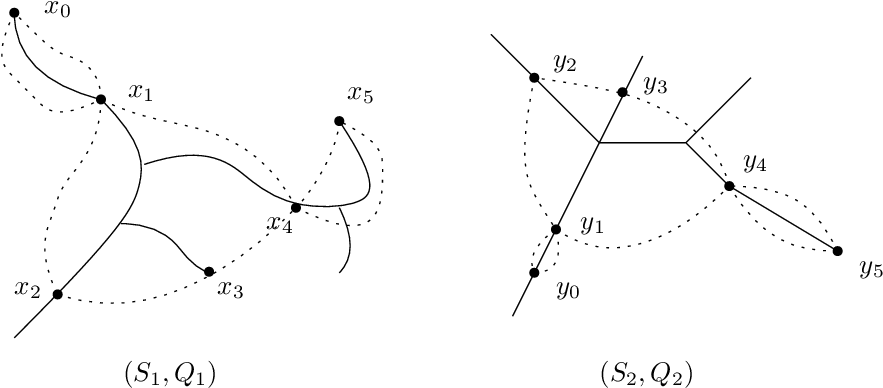}
\caption[fck]{Two pointed trees that should be equivalent. \label{exemples2}}
\end{figure}

This is precisely the suitable definition of our equivalence relation. We will say that two pointed trees $(T,P)$ and $(S,Q)$ are
\emph{spatially equivalent} if there exists a bijection $\map{\phi}{P}[S]$ that preserves discrete components.
Be aware that, in this case, the bijection $\phi$ does not need to be unique, and this fact accounts for the spatial symmetries among pointed trees.
For instance, consider the pointed trees $(T_1,P_1)$ and $(T_3,P_3)$ in Figure~\ref{exemples}. The bijection $\map{\phi}{P_1}[P_3]$ such that $\phi(a)=x$,
$\phi(b)=y$, $\phi(c)=z$, $\phi(d)=t$ preserves discrete components. But the same happens for the bijection $\map{\bar{\phi}}{P_1}[P_3]$ such that $\bar{\phi}(a)=t$,
$\bar{\phi}(b)=z$, $\bar{\phi}(c)=y$, $\bar{\phi}(d)=x$. In this particularly simple example, the sets $P_1$ and $P_3$ are both contained in an interval, and
our definition allows us to do not care about the particular orientation of such intervals.

To see another example of more sophisticated symmetries, consider the trees $(S_1,Q_1)$ and $(S_2,Q_2)$ in Figure~\ref{exemples2}. Here the bijection
$\map{\phi}{Q_1}[Q_2]$ such that $\phi(x_i)=y_i$ for $0\le i\le 5$ preserves discrete components. The same happens for the bijection $\map{\bar{\phi}}{Q_1}[Q_2]$
such that $\bar{\phi}(x_0)=y_5$, $\bar{\phi}(x_1)=y_4$, $\bar{\phi}(x_2)=y_2$, $\bar{\phi}(x_3)=y_3$, $\bar{\phi}(x_4)=y_1$, $\bar{\phi}(x_5)=y_0$,
which transforms the discrete components in this way:
\[ \begin{array}{rcl}
& \bar{\phi} & \\
\{ x_0,x_1 \} & \longrightarrow & \{ y_5,y_4 \} \\
\{ x_4,x_5\} & \longrightarrow & \{ y_1,y_0 \} \\
\{ x_1,x_2,x_3,x_4\} & \longrightarrow & \{ y_4,y_2,y_3,y_1 \}
\end{array} \]

Another interesting symmetry appears via the bijection $\map{\varphi}{Q_1}[Q_2]$
such that $\varphi(x_0)=y_5$, $\varphi(x_1)=y_4$, $\varphi(x_2)=y_3$, $\varphi(x_3)=y_2$, $\varphi(x_4)=y_1$, $\varphi(x_5)=y_0$. Then,
\[ \begin{array}{rcl}
& \varphi & \\
\{ x_0,x_1 \} & \longrightarrow & \{ y_5,y_4 \} \\
\{ x_4,x_5\} & \longrightarrow & \{ y_1,y_0 \} \\
\{ x_1,x_2,x_3,x_4\} & \longrightarrow & \{ y_4,y_3,y_2,y_1 \}
\end{array} \]

From those examples it is clear that what characterizes a class of spatially equivalent pointed trees is just a collection of discrete
components and their connections. In other words, we do not take care of the particular topology of the underlying tree, the particular labelling of the
distinguished points and the spatial symmetries.

Now we are ready to define the notion of a \emph{tree pattern} (or \emph{pattern} for brevity). A triplet $(T,P,f)$ will be called a \emph{model}
if $\map{f}{T}$ is a tree map and $P$ is a finite $f$-invariant set. In particular, if $P$ is a periodic orbit of $f$ then $(T,P,f)$
will be called a \emph{periodic model} (or an \emph{$n$-periodic model} if $|P|=n$ and we want to specify the period). Let $(T,P,f)$ and
$(S,Q,g)$ be two models. We will say that $(T,P,f)$ and $(S,Q,g)$ are \emph{equivalent}, written $(T,P,f)\sim(S,Q,g)$, if and only if:
\begin{enumerate}
\item The pointed trees $(T,P)$ and $(S,Q)$ are spatially equivalent.
\item There exists a bijection $\map{\phi}{P}[Q]$ that preserves discrete components and conjugates the maps $f\evalat{P}$ and $g\evalat{Q}$:
\[ f\evalat{P}=\phi^{-1}\circ g\evalat{Q}\circ\phi. \]
\end{enumerate}

It is easy to see that $\sim$ defines an equivalence relation. A class of models with respect to $\sim$ will be called a \emph{pattern}.
Any representative $(T,P,f)$ of a pattern $\mathcal{P}$ will be said to \emph{exhibit} $\mathcal{P}$, and we will write $\mathcal{P}=[T,P,f]$.
The pattern $\mathcal{P}$ itself will be called \emph{periodic} if a model (in consequence, \emph{any} model) exhibiting $\mathcal{P}$ is periodic.

In Figure~\ref{exemples3} (left, center) we show two periodic models $(S_1,Q_1,f)$ and $(S_2,Q_2,g)$ equivalent by $\sim$, constructed from the
(spatially equivalent) pointed trees $(S_1,Q_1)$ and $(S_2,Q_2)$ in Figure~\ref{exemples2}. We only show the images by $f$ and $g$
of the points in $Q_1$ and $Q_2$, respectively, since this is the only information we need. In
this example, the bijection $\map{\varphi}{Q_1}[Q_2]$ that preserves discrete components and conjugates $f\evalat{Q_1}$ and $g\evalat{Q_2}$
is $\varphi(x_0)=y_5$, $\varphi(x_1)=y_4$, $\varphi(x_2)=y_3$, $\varphi(x_3)=y_2$, $\varphi(x_4)=y_1$, $\varphi(x_5)=y_0$. See a representation of the
corresponding pattern in Figure~\ref{exemples3} (right).

\begin{figure}
\centering
\includegraphics[scale=0.6]{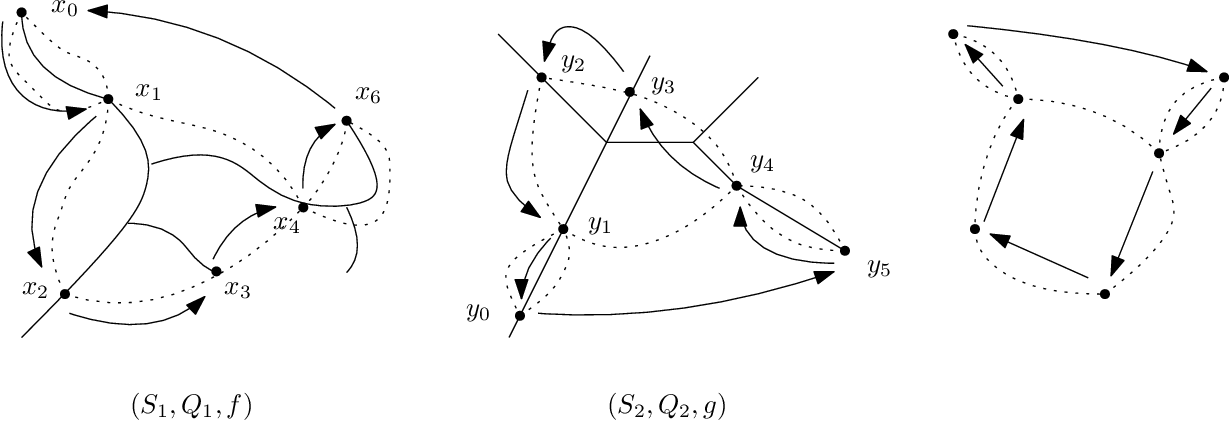}
\caption[fck]{Letf and center: two periodic models exhibiting the same pattern (an arrow from $a$ to $b$ codifies that the image
of $a$ by the corresponding map is $b$). Right: a possible representation of the pattern exhibited by both models. \label{exemples3}}
\end{figure}

An $n$-periodic orbit $P=\{x_i\}_{i=0}^{n-1}$ of a map $\theta$ will be said to be \emph{time labeled} if
$\theta(x_i)=x_{i+1}$ for $0\le i<n-1$ and $\theta(x_{n-1})=x_0$. For instance, the orbit $Q_1$ in Figure~\ref{exemples3}
is time labeled, while the orbit $Q_2$ is not.

Despite the fact that the notion of a discrete component is defined for pointed trees,
by abuse of language we will use the expression \emph{discrete component of a pattern},
which will be understood in the natural way since the number of discrete components
and their relative positions are the same for all models of the pattern. In the same spirit we will use
the expression \emph{point of a pattern}. All patterns considered in this paper will be periodic. By default we will
consider that the points of an $n$-periodic pattern are time labeled with the integers $\{0,1,\ldots,n-1\}$. As an example,
Figure~\ref{exemples4} shows a representation of the 6-periodic pattern of Figure~\ref{exemples3}(right), whose
set of discrete components is $\{\{0,1\},\{4,5\},\{1,2,3,4\}\}$.

In whatever context, a good definition of \emph{pattern} should lead to a theory rich enough to provide relevant information about
the dynamics in the family of maps under consideration. In particular, the theory must be able to give answers to (at least)
the following questions:
\begin{itemize}
\item[(Q1)] For a pattern $\mathcal{P}$, does there exist a number $h(\mathcal{P})$ that measures the dynamical complexity
forced by $\mathcal{P}$, in the sense that $h(\mathcal{P})$ is a lower bound for the dynamical complexity of \emph{any} map exhibiting $\mathcal{P}$?
\item[(Q2)] Is $h(\mathcal{P})$ optimal, in the sense that there is a specially simple map exhibiting $\mathcal{P}$ (let us call it \emph{canonical model of $\mathcal{P}$})
whose dynamical complexity is $h(\mathcal{P})$?
\item[(Q3)] It is possible to construct effectively the canonical model of $\mathcal{P}$ in terms only of the combinatorial data encoded in $\mathcal{P}$?
\item[(Q4)] It is possible to compute effectively the dynamical complexity of the canonical model (and, in consequence, $h(\mathcal{P}$))?
\end{itemize}
A positive answer to the above questions allows us to stratify the set of patterns according to their dynamical complexities. Precisely our definition
of \emph{tree pattern} satisfies this requirement. Let us see it.

\begin{figure}
\centering
\includegraphics[scale=0.55]{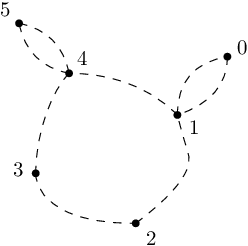}
\caption[fck]{A representation of the 6-periodic pattern in Figure~\ref{exemples3}(right).\label{exemples4}}
\end{figure}

A usual way of measuring the dynamical complexity of a map $\map{f}{X}$ of a compact metric space
is in terms of its \emph{topological entropy} \cite{AKM}. It is a non-negative real number (or
infinity), denoted as $h(f)$, that measures how the iterates of the map mix the points of $X$. For several compact metric
spaces, a map with positive entropy is \emph{chaotic} in the sense of Li and Yorke \cite{li-yorke,blan}. By contrast, the dynamics of a map with zero
topological entropy is essentially trivial.

In view of what have been said above, for a tree pattern $\mathcal{P}$ it is natural to define the \emph{topological entropy of $\mathcal{P}$} as
\begin{equation}\label{hP}
h(\mathcal{P}):=\inf\{h(f):\ (T,P,f)\mbox{ is a model exhibiting }\mathcal{P}\}.
\end{equation}

\begin{remark}[{\bf Standing assumption}]\label{pruned}
For a model $(T,P,f)$, it is well known that the restriction of $f$ to the convex hull of $P$, i.e. $\map{r\circ f}{\chull{P}_T}$ where $\map{r}{T}[\chull{P}_T]$ is the
natural retraction, has entropy smaller than or equal to $h(f)$. So, in the definition (\ref{hP}) it is enough to consider only models $(T,P,f)$ such that the endpoints
of the tree $T$ belong to $P$. In what follows, when we say that $(T,P,f)$ is a model, we implicitly assume that $\En(T)\subset P$.\qed
\end{remark}

Once the topological entropy of a pattern $h(\mathcal{P})$ has been defined (question Q1 above), let us see that there exists a representative of $\mathcal{P}$
whose entropy is precisely $h(\mathcal{P})$, so giving also a positive answer to question Q2.

Let $(T,P,f)$ be a model. We will say that $f$ is \emph{$P$-monotone} if $f\evalat{[a,b]}$ is
monotone (as an interval map) for any pair $\{a,b\}\subset P$ contained in a single discrete component of $(T,P)$.
The model $(T,P,f)$ will then be said to be \emph{monotone}.

The following result is Proposition 4.2 of \cite{aglmm}. It states that if $(T,P,f)$ is a monotone model
then the image of every vertex is uniquely determined and is either a vertex or belongs to $P$.

\begin{proposition}\label{4.2 aglmm}
Let $(T,P,f)$ be a monotone model and let $v\in V(T)\setminus P$. If $a,b,c\in P$ are contained in a single discrete
component and $v\in[a,b]\cap[a,c]\cap[b,c]$, then $f(v)$ is the only point contained in $[f(a),f(b)]\cap[f(a),f(c)]\cap[f(b),f(c)]$.
\end{proposition}

It is implicit in the previous result that the determination of the image $f(v)$ can be done by using \emph{any}
triplet $\{a,b,c\}$ of three closest points of $P$ around $v$. This condition is very restrictive and, a priori, there is no
reason to expect that a pattern $\mathcal{P}$ has a $P$-monotone representative. In fact, in general is it not true that,
\emph{for a given tree $S$}, there is a monotone model $(S,P,f)$ of $\mathcal{P}$. For instance, if $S$ is a 4-star,
it is not possible to have a monotone model $(S,P,f)$ of the pattern shown on Figure~\ref{patintro} (left). To see it,
let us call $y$ the central point of $S$ and assume that $\map{f}{S}$ is $P$-monotone. Using Proposition~\ref{4.2 aglmm}
with $a=1,b=3,c=4$ yields $f(y)\in [2,4]\cap[2,5]\cap[4,5]=\{y\}$ (see Figure~\ref{patintro} (right)), so $f(y)=y$.
Instead, taking $a=0,b=2,c=4$ yields $f(y)\in [1,3]\cap[1,5]\cap[3,5]=\{1\}$, so $f(y)=1$. This contradiction tells us that
the particular topology of a 4-star does not permit the existence of a monotone model of $\mathcal{P}$.
However, since the topology of the underlying tree is not fixed inside the class of models representing $\mathcal{P}$,
the question arises whether there exists another topology permitting it. The answer is affirmative:
Theorem~A of \cite{aglmm} states that for every pattern $\mathcal{P}$ \emph{there exists} a tree $T$ and a map
$\map{f}{T}$ having an invariant set $P$ such that $(T,P,f)$ is a monotone model of $\mathcal{P}$. In
Figure~\ref{patintro} (center) we show a monotone model $(T,P,f)$ of the pattern $\mathcal{P}$,
where $T$ is the tree in the form of a letter ``H''. The reader may check that using Proposition~\ref{4.2 aglmm} with
any combination $\{a,b,c\}$ satisfying the hypothesis yields $f(v)=w$ and $f(w)=1$. Theorem~A of \cite{aglmm} also states
that, for every monotone model $(T,P,f)$ of $\mathcal{P}$, $h(f)=h(\mathcal{P})$.

\begin{figure}
\centering
\includegraphics[scale=0.6]{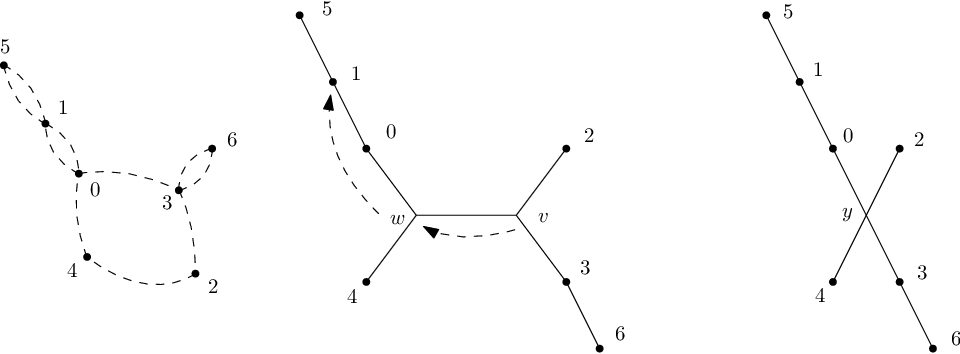}
\caption[fck]{Left: a 7-periodic pattern $\mathcal{P}$. Center: the canonical model of $\mathcal{P}$. Right: the points of
$\mathcal{P}$ embedded on a 4-star.\label{patintro}}
\end{figure}

There exists a special class of monotone models, satisfying several extra minimality properties
that we omit here, called \emph{canonical models}. Theorem~B of \cite{aglmm} states
that every pattern has a canonical model. Moreover, the canonical model of a pattern is essentially unique.
Summarizing, we have the following result.

\begin{theorem}\label{A-aglmm}
Let $\mathcal{P}$ be a pattern. Then the following statements hold.
\begin{enumerate}
\item There exists a canonical model of $\mathcal{P}$.
\item The canonical model $(T,P,f)$ of $\mathcal{P}$ satisfies $h(f)=h(\mathcal{P})$.
\end{enumerate}
\end{theorem}

The proof of Theorem~\ref{A-aglmm} is constructive and gives a finite algorithm to
construct the canonical model of any pattern. Summarizing, questions Q2 and Q3 have a positive answer.

We discuss question Q4 in the next section, since it will bring us the opportunity to explain the scope
of this paper and to present its first main result.

\section{From Topology to Combinatorics. First main result}\label{S2}
Is $h(\mathcal{P})$ computable for any pattern $\mathcal{P}$? The answer is positive. One possibility
is to construct explicitly the canonical model $(T,P,f)$ of $\mathcal{P}$. Now, Proposition~\ref{4.2 aglmm}
states that $P\cup V(T)$ is $f$-invariant. Moreover, since $f$ is monotone between any pair of consecutive
points of $P$, in particular $f$ is monotone on
each connected component of $T\setminus(P\cup V(T))$, which is an interval. Under these conditions, the triplet
$(T,P\cup V(T),f)$ is a particular case of what is called a \emph{Markov model}. Indeed, a model $(T,Q,f)$ is called
\emph{Markov} if $Q\supset V(T)$ and $f$ is monotone on the closure of every connected component of
$T\setminus Q$ (which is called a \emph{$Q$-basic interval}). In this situation it is well known that
we can construct a \emph{Markov matrix} as follows. Consider a labeling $I_1,I_2,\ldots I_k$ of all
$Q$-basic intervals. The \emph{Markov graph of $(T,Q,f)$} is a combinatorial directed graph whose
vertices are the basic intervals and there is an arrow from $I_i$ to $I_j$ if and only if $f(I_i)\supset I_j$.
The associated \emph{Markov matrix $M$} is the corresponding transition matrix, i.e.
$M=(m_{ij})$ with $m_{ij}=1$ if $f(I_i)\supset I_j$ and $m_{ij}=0$ otherwise. The following result is standard
(see \cite{bgmy} or Theorem~4.4.5 of \cite{biblia}).

\begin{proposition}\label{markov}
Let $(T,Q,f)$ be a Markov model and let $M$ be the Markov matrix of $(T,Q,f)$. Then,
$h(f) = \log\max\{\rho(M),1\}$, where $\rho(M)$ denotes the spectral radius of $M$.
\end{proposition}

From Proposition~\ref{markov} and the previous paragraph, the entropy of a pattern $\mathcal{P}$ is computable.
Note that this procedure is ``topological'' rather than ``combinatorial'', since it is based on
the particular topology of the canonical model of $\mathcal{P}$. However, the algorithm to construct the canonical
model is time consuming (the steps to construct the particular tree $T$ from the combinatorial data of the pattern
are far from being trivial). So, one may wonder whether there is a purely combinatorial alternative to compute
$h(\mathcal{P})$. Let us see that this is the case.

Let $(T,P,f)$ be (\emph{any}) model of $\mathcal{P}$. A \emph{basic path} is any pair
$\{a,b\}\subset P$ contained in a discrete component. Let $\{\pi_1,\pi_2,\ldots,\pi_k\}$ be the set
of basic paths of $(T,P)$. We will say that $\pi_i$ \emph{$f$-covers} $\pi_j$,
denoted by $\pi_i\rightarrow\pi_j$, whenever $\pi_j\subset \chull{f(\pi_i)}_T$. The
\emph{$\mathcal{P}$-path graph} is the combinatorial directed graph whose vertices are in one-to-one
correspondence with the basic paths of $(T,P)$, and there is an arrow from the vertex $i$ to the vertex
$j$ if and only if $\pi_i$ $f$-covers $\pi_j$. The associated transition matrix, denoted by $M_\mathcal{P}$,
will be called the \emph{path transition matrix of $\mathcal{P}$}. Note that the definitions of
the $\mathcal{P}$-path graph and the matrix $M_{\mathcal{P}}$ are independent of the particular
choice of the model $(T,P,f)$. Thus, they are well-defined pattern invariants and it is not necessary
to construct the canonical model of $\mathcal{P}$. One of the main results of \cite{aglmm} states that
\begin{equation}\label{patent}
 h(\mathcal{P}) =  \log\max\{\rho(M_{\mathcal{P}}),1\}.
\end{equation}

Once we know that the entropy of a pattern is computable, several natural questions arise.
For instance, it is possible to describe the shape (configuration of discrete components) of the $n$-periodic patterns with
entropy zero \cite{aglmm,reducibility}? For any $n$, which is the $n$-periodic pattern with maximum entropy
\cite{KS,GT,GZ,ajk,ajkm}? And which is the one with minimum positive entropy \cite{ajm,jmr}?

The characterization of patterns of entropy zero was first given in \cite{aglmm} (see also \cite{Blokh}), and another
description was proven to be equivalent in \cite{reducibility}. Both characterizations, however, are more topological than
combinatorial, since they require to construct the canonical model of the pattern. Precisely, the first main result of this
paper is a purely combinatorial description of the periodic patterns of entropy zero. Let us present it.

We start by reviewing the description of zero entropy patterns given in \cite{reducibility}. A pattern will be said to be
\emph{trivial} if it has only one discrete component. Clearly, the entropy of a trivial pattern is zero. The structure of zero entropy
periodic patterns is related to the existence of a \emph{block structure}, a notion that is classic in the field of Combinatorial Dynamics.
In the interval case, the Sharkovskii's \emph{square root construction}
\cite{shar} is an early example of a block structure. The concept of \emph{extension}, first appeared in \cite{block2}, gives rise
to some particular cases of block structures. The notion of \emph{division}, introduced in \cite{lmpy} for interval periodic orbits and
generalized in \cite{ay1} to study the entropy and the periodic structure of tree maps, is also a particular case of block structure.

Let $\mathcal{P}$ be a nontrivial $n$-periodic pattern with $n\ge3$ and let $(T,P,f)$ be (any) model exhibiting $\mathcal{P}$.
For $n>p\ge2$, we will say that $\mathcal{P}$ \emph{has a $p$-block structure} if there exists a partition $P=P_0\cup P_1\cup\ldots\cup P_{p-1}$
such that $f(P_i)=P_{i+1\bmod p}$ for $i\ge0$, and $\chull{P_i}_T\cap P_j=\emptyset$ for $i\ne j$ (from the equivalence relation which defines
the class of models exhibiting $\mathcal{P}$, it easily follows that this condition is independent of the model representing $\mathcal{P}$).
In this case, $p$ is a strict divisor of $n$ and $|P_i|=n/p$ for $0\le i<p$. The sets $P_i$ will be called \emph{blocks}, and the blocks will be said to be
\emph{trivial} if each $P_i$ is contained in a single discrete component of $\mathcal{P}$ (equivalently, the pattern
$[\chull{P_i}_T,P_i,f^p]$ is trivial for any $i$). Note that $\mathcal{P}$ can have several block structures, but only one $p$-block structure for
any given divisor $p$ of $n$.

\begin{remark}[{\bf Standing convention}]\label{standing2}
Let $\mathcal{P}$ be an $n$-periodic pattern whose points are time labeled as $\{0,1,\ldots,n-1\}$. When $\mathcal{P}$ has
a block structure of $p$ blocks $P_0\cup P_1\cup\ldots\cup P_{p-1}$, by convention we will always assume that the time labels
of the blocks have been chosen in such a way that $0\in P_0$.\qed
\end{remark}

In Figure~\ref{ull2blocksV2} we show an example of a 8-periodic pattern $\mathcal{P}$ admitting
a 4-block structure given by $P_0=\{0,4\}$, $P_1=\{1,5\}$, $P_2=\{2,6\}$, $P_3=\{3,7\}$
and also a 2-block structure given by $Q_0=\{0,2,4,6\}$, $Q_1=\{1,3,5,7\}$. In both cases, the blocks are trivial.

\begin{figure}
\centering
\includegraphics[scale=0.65]{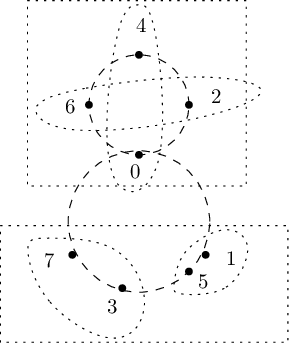}
\caption[fck]{An 8-periodic pattern $\mathcal{P}$ admitting two block structures of trivial blocks.
\label{ull2blocksV2}}
\end{figure}

Let $\mathcal{P}$ be an $n$-periodic pattern with a block structure $P_0\cup P_1\cup\cdots P_{p-1}$. If $(T,P,f)$ is the canonical
model of $\mathcal{P}$ and $\chull{P_i}_T\cap\chull{P_j}_T=\emptyset$ whenever $i\ne j$, then we will say
that the structure is \emph{separated}.

Let $(T,P,f)$ be the canonical model of a pattern $\mathcal{P}$. Let $P=P_0\cup P_1\cup\ldots\cup
P_{p-1}$ be a separated $p$-block structure for $\mathcal{P}$. Then,
$f(\chull{P_i})=\chull{P_{i+1\bmod p}}$. The \emph{skeleton of $\mathcal{P}$}
is the $p$-periodic pattern $\mathcal{S}$ obtained by collapsing each convex hull $\chull{P_i}$
to a point. More precisely, consider the tree $S$ obtained from $T$ by collapsing each tree
$\chull{P_i}$ to a point $x_i$. Let $\map{\kappa}{T}[S]$ be the
standard projection, which is bijective on $T\setminus\cup_i
\chull{P_i}$ and satisfies $\kappa(\chull{P_i})=x_i$. Now set
$Q:=\kappa(P)=\{x_0,x_1,\ldots,x_{p-1}\}$ and define $\map{\theta}{Q}$ by
$\theta(x_i)=x_{i+1\bmod p}$. The $p$-periodic pattern $[S,Q,\theta]$ will be called
the \emph{skeleton} of $\mathcal{P}$.

As an example of construction of a skeleton, consider the 8-periodic pattern consisting of two discrete components
$\{0,2,6\}$, $\{0,1,3,4,5,7\}$ (Figure~\ref{topcol2}, left). Then, $P_0=\{0,4\}$, $P_1=\{1,5\}$, $P_2=\{2,6\}$, $P_3=\{3,7\}$ defines a
structure of 4 trivial blocks. After constructing the canonical model $(T,P,f)$, which is shown in Figure~\ref{topcol2} (top, center),
one gets that $\chull{P_i}_T\cap\chull{P_j}_T=\emptyset$ when $i\ne j$. Thus, the structure is separated and it makes sense to consider
the corresponding 4-periodic skeleton $\mathcal{S}$, shown in Figure~\ref{topcol2} (top, right).

The entropies of a pattern $\mathcal{P}$ with a separated structure of trivial blocks and its associated
skeleton coincide, as the following result (a reformulation of Proposition~8.1 of \cite{aglmm}) states.

\begin{proposition}\label{81}
Let $\mathcal{P}$ be a pattern with a separated structure of trivial blocks and let $\mathcal{S}$ be the associated
skeleton. Then, $h(\mathcal{S})=h(\mathcal{P})$.
\end{proposition}

Now we are ready to state the known recursive characterization of zero entropy patterns in terms of skeletons. The following result
is Proposition~5.6 of \cite{reducibility}.

\begin{proposition}\label{56}
Let $\mathcal{P}$ be a nontrivial periodic pattern. Then, $h(\mathcal{P})=0$ if and only if $\mathcal{P}$ has a separated structure
of trivial blocks such that the associated skeleton has entropy $0$.
\end{proposition}

We stress the fact that Proposition~\ref{56} relies on the notion of separability
of a block structure, which has to be tested over a particular topological representative of the pattern (the canonical model).
The aim of this paper is to give an alternative characterization that does not require the construction of any
particular topological realization --a purely combinatorial criterion.

Assume that a periodic pattern has structures of trivial blocks. The one with blocks with maximum cardinality will be called
the \emph{maximal structure}. For instance, $\{0,2,4,6\}\cup\{1,3,5,7\}$ is the maximal structure for the pattern shown in
Figure~\ref{ull2blocksV2}.

\begin{defi}\label{combicolla}
Let $\mathcal{P}=[T,P,f]$ be a nontrivial $n$-periodic pattern with structures of trivial blocks. Let $P_0\cup P_1\cup\ldots\cup P_{p-1}$ be the maximal
structure. A $p$-periodic pattern $\mathcal{C}=[S,Q,g]$ will be called the \emph{combinatorial collapse of $\mathcal{P}$} if the following properties
are satisfied:
\begin{enumerate}
\item[(a)] $g(i)=j$ if and only if $f(P_i)=P_j$
\item[(b)] For any $0\le i<j\le p-1$, there is a discrete component of $\mathcal{C}$ containing the points $i,j$ if and only if
there is a discrete component of $\mathcal{P}$ intersecting the blocks $P_i,P_j$.
\end{enumerate}
Note that the combinatorial collapse of a pattern with structures of trivial blocks is unique, since it is always carried
out over the maximal structure, which is unique by definition. As an example, see Figure~\ref{omg}.
\end{defi}

\begin{figure}
\centering
\includegraphics[scale=0.5]{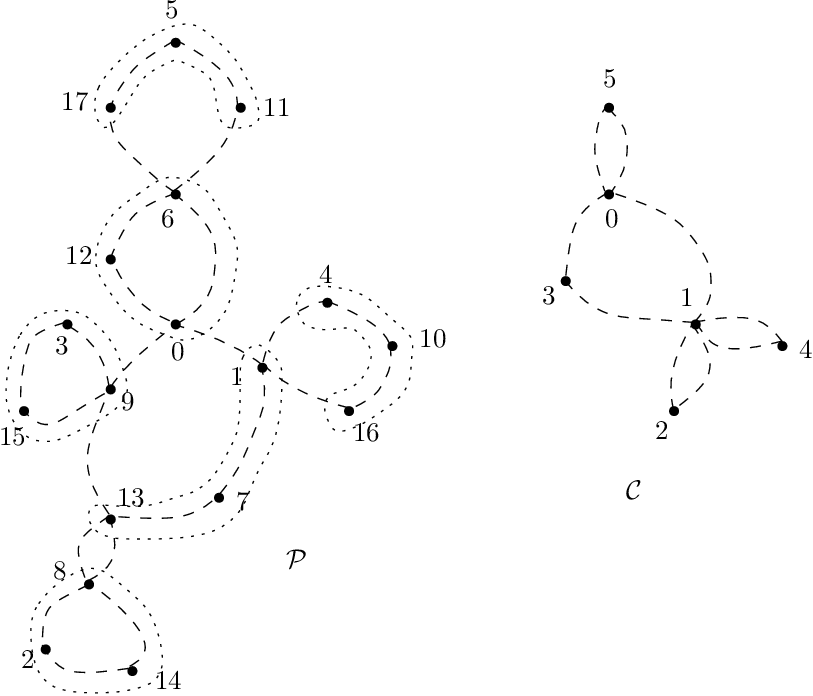}
\caption[fck]{A 18-periodic pattern $\mathcal{P}$ with a maximal structure of 6 blocks, and the combinatorial collapse $\mathcal{C}$.\label{omg}}
\end{figure}

Observe the strong analogy between the concepts of combinatorial collapse and skeleton. However, in general both patterns differ
from each other. See Figure~\ref{topcol2}.

\begin{figure}
\centering
\includegraphics[scale=0.55]{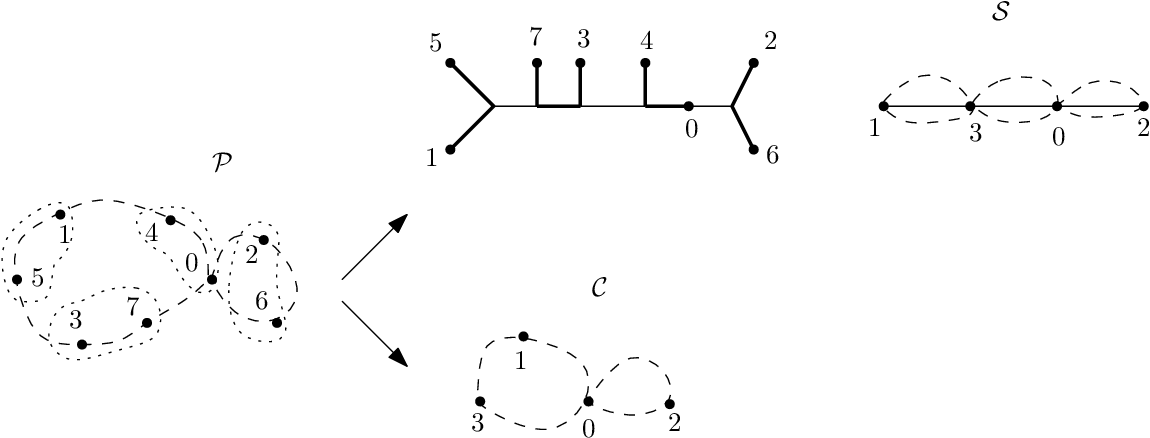}
\caption[fck]{Left: an 8-periodic pattern $\mathcal{P}$ with a maximal and separated structure of 4 blocks. Up, center:
the canonical model $(T,P,f)$ of $\mathcal{P}$, the convex hulls of the blocks marked with thick lines. Up, right:
the corresponding skeleton $\mathcal{S}$. Bottom: the combinatorial collapse $\mathcal{C}$. \label{topcol2}}
\end{figure}

The following proposition, a kind of combinatorial analog of Proposition~\ref{81}, is the core result of this paper.

\begin{MainProposition}\label{SC}
Let $\mathcal{P}$ be a nontrivial periodic pattern having structures of trivial blocks and let $\mathcal{C}$ be the
combinatorial collapse of $\mathcal{P}$. Then, $h(\mathcal{C})=h(\mathcal{P})$.
\end{MainProposition}

Proposition~\ref{SC} will be proved in Section~\ref{S5}. Using Propositions~\ref{81} and \ref{56}, it has the following immediate
consequence.

\begin{MainProposition}\label{mainA}
Let $\mathcal{P}$ be a nontrivial periodic pattern. Then, $h(\mathcal{P})=0$ if and only if $\mathcal{P}$ has structures of trivial
blocks and the combinatorial collapse of $\mathcal{P}$ has entropy $0$.
\end{MainProposition}

Observe the recursive nature of Proposition~\ref{mainA}: if $\mathcal{P}$ is not a trivial pattern, then its combinatorial collapse, with entropy zero
and a period that strictly divides that of $\mathcal{P}$, is either trivial or has also structures of trivial blocks with a combinatorial collapse
of entropy zero. We can thus iterate the process as many times as necessary to finally obtain a trivial pattern, suggesting the following
definition.

\begin{defi}\label{explosions}
We say that an $n$-periodic pattern $\mathcal{P}$ is \emph{strongly collapsible} if there exist a sequence
of patterns $\{\mathcal{P}_i\}_{i=0}^r$ and a sequence of integers $\{p_i\}_{i=0}^r$ for some $r\ge0$ such that:
\begin{enumerate}
\item $\mathcal{P}_r=\mathcal{P}$
\item $\mathcal{P}_0$ is a trivial $p_0$-periodic pattern
\item For $1\le i\le r$, $\mathcal{P}_{i}$ has a maximal structure of $\prod_{j=0}^{i-1} p_{j}$ trivial blocks of cardinality $p_i$
and $\mathcal{P}_{i-1}$ is the combinatorial collapse of $\mathcal{P}_i$.
\end{enumerate}

The sequence $\{\mathcal{P}_i\}_{i=0}^r$ will be called \emph{the sequence of collapses of} $\mathcal{P}$. Notice that $\prod_{j=0}^{r}p_j=n$.
See Figure~\ref{ullh0} for an example with $p_0=3$, $p_1=2$, $p_2=3$.
\end{defi}

\begin{remark}\label{captrivial}
It is implicit in Definition~\ref{explosions} that we stop the collapsing process as far as we get a trivial pattern. In other words, in
a sequence of collapses $\{\mathcal{P}_i\}_{i=0}^r$, all patterns $\mathcal{P}_i$ with $1\le i\le r$ have at least two discrete components.\qed
\end{remark}

As an immediate consequence of Proposition~\ref{mainA} we get the following corollary, which is in fact the first main result of this paper.

\begin{MainTheorem}\label{mainB}
Let $\mathcal{P}$ be a periodic pattern. Then, $h(\mathcal{P})=0$ if and only if $\mathcal{P}$ is strongly collapsible.
\end{MainTheorem}

\section{An application to star maps. Second main result}\label{S3}
As an example of application of Theorem~\ref{mainB}, we look now at a classic problem in Combinatorial Dynamics. Let us fix a tree $T$.
We want to know for which values of $n$ it is possible to have a zero entropy self-map of $T$ having an $n$-periodic orbit. A more
ambitious problem is to determine, for any prescribed period $n$, an optimal lower bound for the topological entropy of any self-map of $T$ having
an $n$-periodic orbit. More precisely, for any integer $n\ge 1$, define $M_n(T)$ (standing for \emph{minimum entropy for period $n$}) as
\[ M_n(T):=\inf\{h(f): (T,P,f)\mbox{ is an $n$-periodic model}\}. \]

We want to know for which values of $n$ we have $M_n(T)=0$. The ambitious version of the problem is to find a closed formula for $M_n(T)$ in
terms of $n$. Of course, the answer to both questions depends strongly on the particular tree $T$. These problems are in general very difficult and
have been solved only when $T$ is an interval \cite{bgmy} or a 3-star \cite{almy,YeXY}. For some more general families of trees, partial answers to
these and other related questions are given in \cite{j1,j2}.

\begin{figure}
\centering
\includegraphics[scale=0.75]{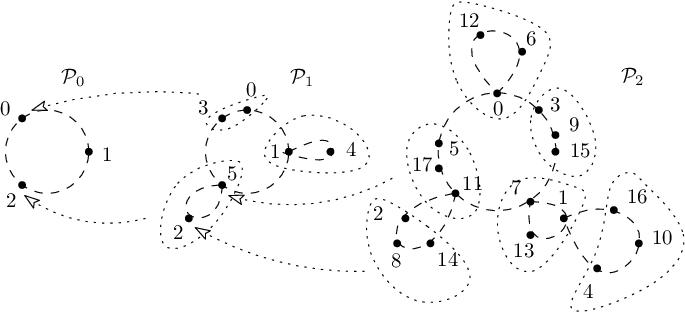}
\caption[fck]{An example of a zero entropy 18-periodic pattern $\mathcal{P}_2$ and the corresponding sequence of collapses.\label{ullh0}}
\end{figure}

The second main result of this paper describes for which periods $n$ we have $M_n(T)=0$ when $T$ is a star. For $n\ge k\ge 2$, set
\[ M_{n,k}:=\inf\{h(f): (T,P,f)\mbox{ is an $n$-periodic model such that $T$ is a $k$-star}\}. \]

In the previous definition we are taking \emph{models}. Recall (Remark~\ref{pruned}) that, by definition, if $(T,P,f)$ is a model then $\En(T)\subset P$.
So, $M_{n,k}$ takes into account only periodic orbits \emph{properly} contained in $T$, i.e. having at least one point in each branch of the star.

The following is the second main result of this paper, and its proof uses the combinatorial characterization of zero entropy tree patterns given by
Theorem~\ref{mainB}.

\begin{MainTheorem}\label{mainC}
Let $n,k$ be integers such that $n\ge k\ge 3$. Then, $M_{n,k}=0$ if and only if:
\begin{enumerate}
\item Either $n=k\cdot 2^q$ with $q\ge0$
\item or $n=2^q$ with $q\ge k-1$.
\end{enumerate}
\end{MainTheorem}

For example, there exists a zero entropy map on a 4-star having an 8-periodic orbit (with points in all branches), while this is
not possible on a 5-star.

Some readers may find not natural to restrict our attention only to periodic orbits properly contained in the star. If this is the case, then drop this
condition and simply define

\[ \overline{M}_{n,k}:=\inf\{h(f): \map{f}{T}\mbox{ has an $n$-periodic orbit and $T$ is a $k$-star}\}. \]

It is easy to derive a result analog to Theorem~\ref{mainC} for $\overline{M}_{n,k}$ instead of $M_{n,k}$. The reason is that any continuous map
of a star can be trivially extended to a bigger star with the same entropy. More precisely, assume that $S\subset S'$, where $S$ is an $\ell$-star and $S'$
is a $k$-star with $\ell<k$. Let $y$ be the common central point of both stars. Consider a periodic orbit $P$ of $\map{f}{S}$ such that $\chull{P}_S=S$. The
\emph{trivial extension of $f$ to $S'$} is a map $\map{f'}{S'}$ such that $f'\evalat{S}=f\evalat{S}$ and $f'$ maps identically $S'\setminus S$ onto $f(y)$.
Then, $f'\evalat{P}=f\evalat{P}$ and $h(f')=h(f)$.

Recall that $M_{n,k}$ is well defined only when $n\ge k$, while $\overline{M}_{n,k}$ makes sense for any pair $n,k$. When $n<k$, a rigid rotation 
of an $n$-star is an example of a zero entropy map with a periodic orbit of period $n$. Considering its trivial extension to a $k$-star, we get
\[ \overline{M}_{n,k}=0\mbox{ when }1\le n<k. \]
On the other hand, the technique of trivial extensions tells us that
\[ \overline{M}_{n,k} = \min\{ M_{n,\ell}: 2\le\ell \le k\}\mbox{ when }n\ge k. \]

From the previous discussion, we get the following immediate consequence of Theorem~\ref{mainC}.

\begin{MainCorollary}
Let $n,k$ be positive integers such that $k\ge 3$. Then, $\overline{M}_{n,k}=0$ if and only if $n=\ell\cdot 2^q$ with $q\ge0$ and $1\le\ell\le k$.
\end{MainCorollary}

\section{Proof of Proposition~\ref{SC}}\label{S5}
A key ingredient to prove Proposition~\ref{SC} is a tool, introduced in \cite{ajm}, that allows us to compare the entropies of two patterns under certain conditions.
For the sake of brevity, here we will give a somewhat informal (though completely clear) version of this procedure.

Let $(T,P,f)$ be a model of a pattern $\mathcal{P}$. We recall that two discrete components of $(T,P)$ are either disjoint or intersect at a single point of $P$. Two discrete
components $A,B$ of $(T,P)$ will be said to be \emph{adjacent at $x\in P$} (or simply \emph{adjacent}) if $A\cap B=\{x\}$.
If we join together two adjacent discrete components $A$ and $B$ to get a new discrete component $A\cup B$ and keep intact the remaining components, we get a new pattern $\mathcal{O}$. We will say that $\mathcal{O}$ is a \emph{2-opening of $\mathcal{P}$}. As an example, see Figure~\ref{brutty2}, where $\mathcal{O}$ is a 2-opening of $\mathcal{P}$ obtained by joining the discrete components $A=\{2,5,6\}$ and $B=\{0,5\}$, while $\mathcal{R}$ is a 2-opening of $\mathcal{P}$ obtained by joining $B$ and $C=\{1,3,5\}$.

The previous procedure can be repeated in order to join together more discrete components. We will say that $\mathcal{Q}$ is an \emph{opening} of $\mathcal{P}$ if there exist $l\ge 2$ patterns $(\mathcal{P}_i)_{i=1}^l$ such that
$\mathcal{P}_1=\mathcal{P}$, $\mathcal{P}_l=\mathcal{Q}$, and $\mathcal{P}_{i+1}$ is a 2-opening of $\mathcal{P}_i$. In Figure~\ref{brutty2}, $\mathcal{O}'$ and $\mathcal{R}'$ are two different openings of $\mathcal{P}$.

\begin{figure}
\centering
\includegraphics[scale=0.65]{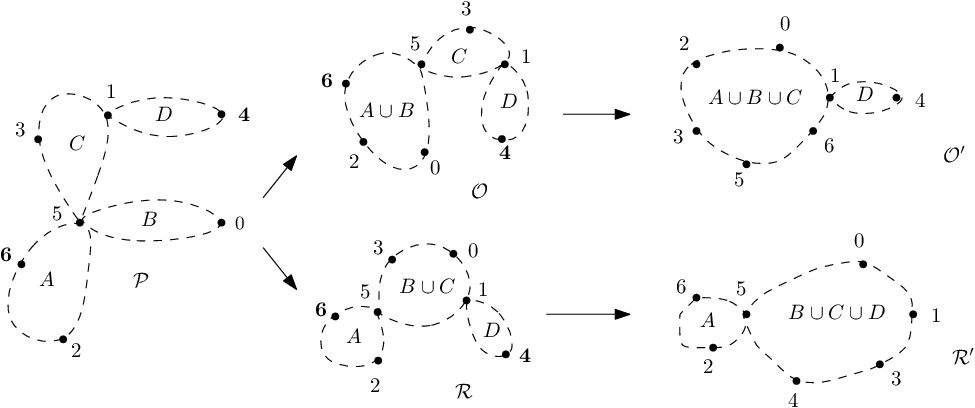}
\caption[fck]{$\mathcal{O}$ and $\mathcal{R}$ are two different 2-openings of $\mathcal{P}$, and $\mathcal{O}'$ and
$\mathcal{R}'$ are two different openings of $\mathcal{P}$. \label{brutty2}}
\end{figure}

As it is clear from the examples shown in Figure~\ref{brutty2}, we are implicitly assuming that the labeling of the points
of an $n$-periodic pattern $\mathcal{P}$ fixes the labeling of the points of any opening of $\mathcal{P}$.

As one may expect from intuition (we are dropping restrictions among points), the entropy of a model decreases when performing an opening,
as the following result (Theorem~5.3 of \cite{ajm}) states.

\begin{theorem}\label{5.3}
Let $\mathcal{P}$ and $\mathcal{O}$ be periodic patterns. If $\mathcal{O}$ is an opening of $\mathcal{P}$, then $h(\mathcal{P})\ge h(\mathcal{O})$.
\end{theorem}

\begin{remark}\label{skcc}
Recall (Figure~\ref{topcol2}) that the skeleton and the combinatorial collapse of a pattern with a separated structure of trivial blocks
need not coincide. It is clear, however, that the combinatorial collapse is an opening of the skeleton.\qed
\end{remark}

Recall that, when $\mathcal{P}$ has block structures, the one with blocks with maximum cardinality is unique and is called
a \emph{maximal structure}. The following result was proven (but not implicitly stated) in \cite{reducibility}.

\begin{lemma}\label{maxi}
Let $\mathcal{P}$ be a nontrivial periodic pattern having structures of trivial blocks. Then, the maximal structure is separated.
\end{lemma}
\begin{proof}
Let $(T,P,f)$ be the canonical model of $\mathcal{P}$. Let $P_0\cup P_1\cup\cdots\cup P_{p-1}$ be the maximal structure.
Assume, by way of contradiction, that this structure is not separated. In the proof of Lemma~4.1 of \cite{reducibility} it is
shown that, in this case, there exists a divisor $m\ge1$ of $p$ and an $m$-periodic orbit $\{y,f(y),\ldots,f^{m-1}(y)\}\subset
V(T)\setminus P$ such that
\begin{equation}\label{jjj}
f^i(y) \in \bigcap_{j=0}^{\frac{p}{m}-1}\chull{P_{i+jm}}
   \text{ for } 0\le i\le m-1.
\end{equation}
If $m=1$, then (\ref{jjj}) implies that the convex hulls of all blocks $P_i$ intersect at a single point $y$. Since the blocks are trivial,
this implies that all blocks are contained in a unique discrete component and $\mathcal{P}$ is, in consequence, a trivial pattern, in contradiction
with the hypotheses. On the other side, if $m>1$ then we set
\[
  \overline{P}_i:=\bigcup_{j=0}^{\frac{p}{m}-1} P_{i+jm}\,
    \text{ for } 0\le i\le m-1
\]
and observe that $\overline{P}_0\cup\overline{P}_1\cup\cdots\cup\overline{P}_{m-1}$ defines an $m$-block structure for $\mathcal{P}$.
Moreover, the fact that the blocks $P_i$ are trivial implies that each block $\overline{P}_i$ is trivial, in contradiction with the maximality of the
structure $P_0\cup P_1\cup\cdots\cup P_{p-1}$.
\end{proof}

Finally we state the last ingredient we need to prove Proposition~\ref{SC}.

\begin{lemma}\label{basic}
Let $\mathcal{P}$ be a periodic pattern having structures of trivial blocks and let $\mathcal{C}$ be the combinatorial collapse of $\mathcal{P}$.
Then, there exists a pattern $\mathcal{P}'$ such that $h(\mathcal{P}')=h(\mathcal{C})$ and $\mathcal{P}$ is either $\mathcal{P}'$ or an opening of $\mathcal{P}'$.
\end{lemma}
\begin{proof}
Let $n$ be the period of $\mathcal{P}$ and let $p$ be the number of blocks of the maximal structure. All $n$-periodic patterns considered in this proof will be time labeled
with integers $\{0,1,\ldots,n-1\}$. We recall also the standing convention in Remark~\ref{standing2}, so the maximal structure of $\mathcal{P}$ is given by the partition
$K_0\cup K_1\cup\cdots\cup K_{p-1}$, with
\[ K_i:=\{i,i+p,i+2p,\ldots,i+(k-1)p\}, \mbox {with }k=n/p. \]

We will construct the prescribed pattern $\mathcal{P}'$ in exactly $p$ steps.
More precisely, we will define a sequence of patterns $(\mathcal{P}_i)_{i=0}^{p-1}$ satisfying:
\begin{enumerate}
\item either $\mathcal{P}_0=\mathcal{P}$ or $\mathcal{P}$ is a 2-opening of $\mathcal{P}_0$
\item for any $1\le i\le p-1$, either $\mathcal{P}_i=\mathcal{P}_{i-1}$ or $\mathcal{P}_{i-1}$ is a 2-opening of $\mathcal{P}_i$
\item for any $0\le i\le p-1$, $K_0\cup K_1\cup\cdots\cup K_{p-1}$ is a maximal structure for $\mathcal{P}_i$ and the
combinatorial collapse of $\mathcal{P}_i$ is $\mathcal{C}$.
\item $h(\mathcal{P}_{p-1})=h(\mathcal{C})$.
\end{enumerate}
It is clear, then, that the lemma will be proven if we take $\mathcal{P}':=\mathcal{P}_{p-1}$.

Let us start the iterative construction of the sequence of patterns. We advise the reader to follow the process over a particular example,
for instance over the 18-periodic pattern shown in Figure~\ref{omg}, that has a maximal structure of 6 blocks of 3 points.
The corresponding combinatorial collapse is shown in the same picture.

Set $\mathcal{P}_{-1}:=\mathcal{P}$. Now let $i\in\{0,1,\ldots,p-1\}$ and assume that $K_0\cup K_1\cup\cdots\cup K_{p-1}$ is a maximal structure
for $\mathcal{P}_{i-1}$ (this is obviously the case for $i=0$). Next we list the two rules to get $\mathcal{P}_i$
from $\mathcal{P}_{i-1}$ depending on two cases:
\begin{enumerate}
\item[Rule A:] Assume that the block $K_i$ is a full discrete component of $\mathcal{P}_{i-1}$. In this case, we set $\mathcal{P}_i:=\mathcal{P}_{i-1}$.
\item[Rule B:] Assume that the block $K_i$ is strictly contained in a discrete component of $\mathcal{P}_{i-1}$. Let $C_1,C_2,\ldots,C_m$ be the discrete
components of $\mathcal{P}_{i-1}$ and assume without loss of generality that the block $K_i$ is contained in $C_1$. In this case, define the pattern $\mathcal{P}_i$
as the one whose discrete components are
\[ K_i,(C_1\setminus K_i)\cup \{i\},C_2,C_3,\ldots,C_m. \]
\end{enumerate}

As an example of the previous iterative procedure, see Figure~\ref{omg2}, where we show the 6 steps corresponding to the pattern $\mathcal{P}$
shown in Figure~\ref{omg}. In step 0, the block $K_0=\{0,6,12\}$ is a full discrete component of $\mathcal{P}_{-1}=\mathcal{P}$, so we apply Rule A
and we simply take $\mathcal{P}_0:=\mathcal{P}$. Now, at step 1, the block $K_1:=\{1,7,13\}$ is strictly contained in the discrete component
$C=\{0,1,7,9,13\}$ of $\mathcal{P}_0$. So, we apply Rule B and $\mathcal{P}_1$ is obtained after replacing $C$ by the two discrete components
$K_1,(C\setminus K_1)\cup\{1\}=\{1,7,13\},\{1,0,9\}$ and keeping intact the remaining components. And so on.

\begin{figure}
\centering
\includegraphics[scale=0.5]{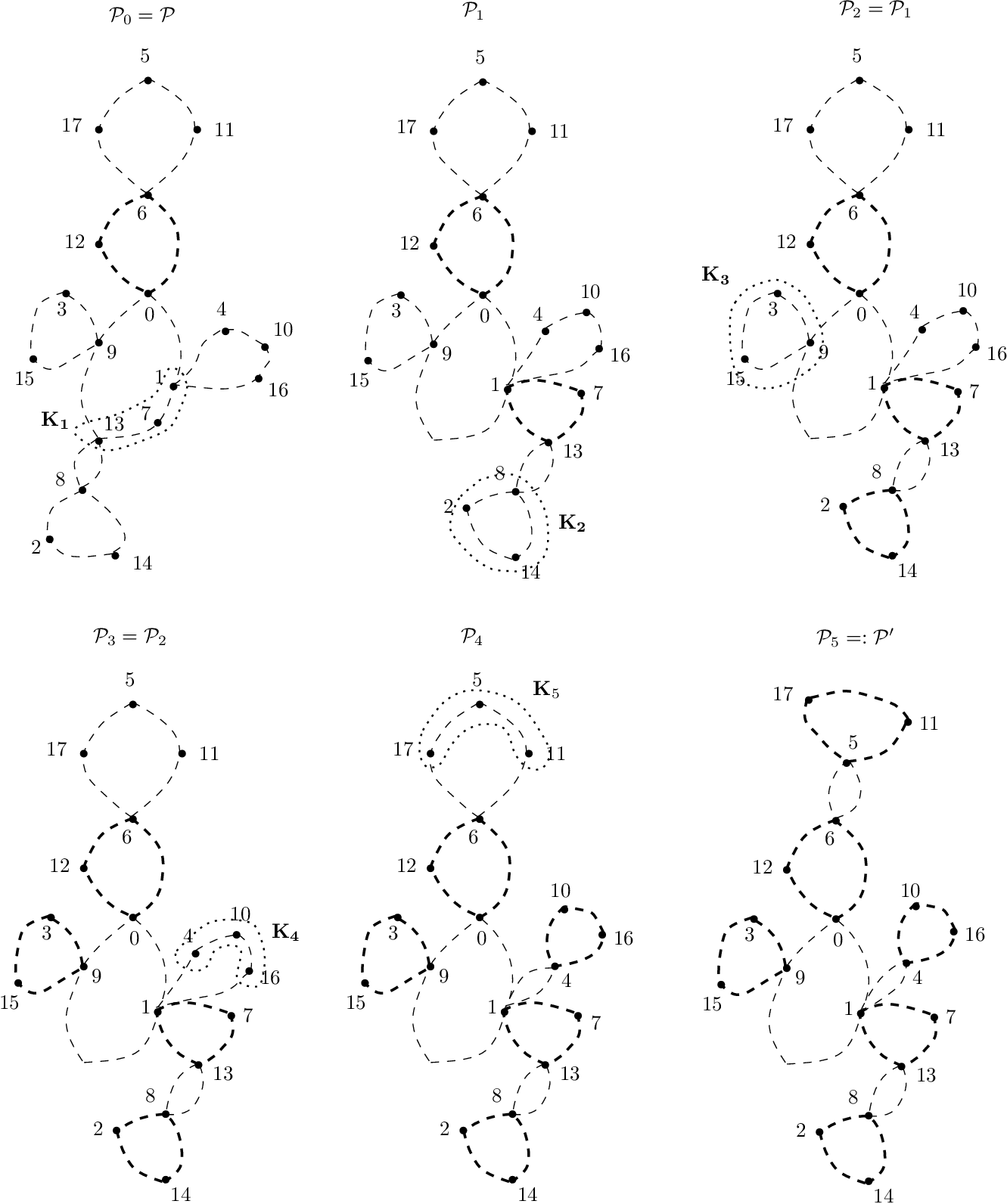}
\caption[fck]{Construction of pattern $\mathcal{P}'$ of Lemma~\ref{basic} corresponding to the pattern $\mathcal{P}$ of Figure~\ref{omg}.\label{omg2}}
\end{figure}

Observe that, when using Rule B, we are separating the block $K_i$ as a new full discrete component. Since $i\in K_i$, the two new discrete components,
$K_i$ and $(C_1\setminus K_i)\cup \{i\}$, are adjacent at the common point $i$. In other words, if we join them together as
discrete components of $\mathcal{P}_i$, then we recover the pattern $\mathcal{P}_{i-1}$. Thus, $\mathcal{P}_{i-1}$ is a 2-opening
of $\mathcal{P}_i$ and properties (a,b) above hold.

Property (c) is easily derived by induction. Indeed, let $0\le i<p$ and assume that $K_0\cup K_1\cup\cdots\cup K_{p-1}$ is a maximal structure
for $\mathcal{P}_{i-1}$, and that $\mathcal{C}$ is the combinatorial collapse of $\mathcal{P}_{i-1}$ ($\mathcal{P}_0$ satisfies these assumptions
since $\mathcal{P}_{-1}=\mathcal{P}$).

We want to see that $K_0\cup K_1\cup\cdots\cup K_{p-1}$ is also a maximal structure for $\mathcal{P}_i$, and that the
combinatorial collapse of $\mathcal{P}_i$ is $\mathcal{C}$. This is obvious if we have applied Rule A to obtain $\mathcal{P}_i$ from
$\mathcal{P}_{i-1}$, so assume that we have applied Rule B.

Let $(T,P,f)$ be any model of $\mathcal{P}_i$. It is evident that $f(K_j)=K_{j+1\bmod p}$. Moreover, since the blocks are trivial in $\mathcal{P}_{i-1}$,
each $K_j$ is contained in a discrete component of $\mathcal{P}_{i-1}$. Together with Rule B, this implies that each set $K_j$ is contained in a
discrete component of $\mathcal{P}_i$. In consequence, $\chull{K_j}_T\cap K_k=\emptyset$ whenever $j\ne k$
and $K_0\cup K_1\cup\cdots\cup K_{p-1}$ is also a maximal block structure for $\mathcal{P}_i$.

Rule B also implies that if there is a discrete component of $\mathcal{P}_{i-1}$ intersecting two blocks $K_j,K_k$, there is also
a discrete component of $\mathcal{P}_i$ intersecting $K_j$ and $K_k$. So, Definition \ref{combicolla} applies and $\mathcal{C}$
is the combinatorial collapse of $\mathcal{P}_i$.

Summarizing, the sequence $(\mathcal{P}_i)_{i=0}^{p-1}$ satisfies properties (a-c) above. Let us see that property (d) is
also satisfied. Let $\mathcal{S}$ be the skeleton of $\mathcal{P}'$ associated to the maximal structure $K_0\cup K_1\cup\cdots\cup K_{p-1}$
which, by Lemma~\ref{maxi}, is separated. Now, Proposition~\ref{81} tells us that $h(\mathcal{S})=h(\mathcal{P}')$. Finally, since
each block $K_i$ is in fact a full discrete component, it follows that the skeleton $\mathcal{S}$ coincides with the combinatorial
collapse $\mathcal{C}$. So, property (d) holds.
\end{proof}

Now we are ready to prove Proposition~\ref{SC}.

\begin{proof}[Proof of Proposition~\ref{SC}]
Take the maximal block structure of $\mathcal{P}$. By Lemma~\ref{maxi}, it is separated. Thus, it makes sense to consider the associated skeleton $\mathcal{S}$.
By Remark~\ref{skcc}, $\mathcal{C}$ is an opening of $\mathcal{S}$. Then, by Theorem~\ref{5.3}, $h(\mathcal{C})\le h(\mathcal{S})$, which equals $h(\mathcal{P})$ by Proposition~\ref{81}.
Now let $\mathcal{P}'$ be the pattern given by Lemma~\ref{basic}, in such a way that $h(\mathcal{P}')=h(\mathcal{C})$ and $\mathcal{P}$ is either $\mathcal{P}'$
or an opening of $\mathcal{P}'$. Using again Theorem~\ref{5.3} yields $h(\mathcal{P})\le h(\mathcal{P}')$. Summarizing,
\[ h(\mathcal{C})\le h(\mathcal{S})=h(\mathcal{P})\le h(\mathcal{P}')=h(\mathcal{C}). \]
\end{proof}

\section{Construction of star patterns of entropy zero}\label{S6}
Although a point is an element of a topological space and a pattern is a combinatorial object defined as an equivalence class of pointed trees,
recall that by abuse of language we talk about the \emph{points} of a pattern. The same translation from topology to combinatorics can be
applied to the term \emph{valence}. The (combinatorial) \emph{valence} of a point $x$ of a pattern $\mathcal{P}$ is defined as the number of
discrete components of $\mathcal{P}$ containing $x$. Note that if $x$ has combinatorial valence $\nu$, then for any model
$(T,P,f)$ of $\mathcal{P}$, the (topological) valence of the point of $T$ corresponding to $x$ is the same and equals $\nu$. So, in what follows
we will drop the words \emph{combinatorial} and \emph{topological} and will use the term \emph{valence} indistinctly in both senses. Let $x$ be a
point of $\mathcal{P}$. If $x$ has valence 1, it will be called an \emph{endpoint}.

A pattern will be called \emph{simplicial} if all its discrete components have 2 points. See Figure~\ref{exstarrys} (right) for an example of a 7-periodic
simplicial pattern. Note that an interval pattern is simplicial.

Let $k\ge 2$. A pattern $\mathcal{P}$ will be called a \emph{$k$-star pattern} (or simply a \emph{star pattern}) if there exists a model $(T,P,f)$ of
$\mathcal{P}$ such that $T$ is a $k$-star. Since $\En(T)\subset P$ by definition of a model, it follows that a $k$-star pattern has at least $k$ points.
Note that if $k=2$, then $\mathcal{P}$ is simplicial. On the other hand, if $k\ge 3$ then:
\begin{enumerate}
\item[(NS)] If $\mathcal{P}$ is non-simplicial, all points of $\mathcal{P}$ have valence 1 or 2 and there is one discrete component of $k$ points.
\item[(S)] If $\mathcal{P}$ if simplicial, all points of $\mathcal{P}$ have valence 1 or 2 except one point, that has valence $k$.
\end{enumerate}

See Figure~\ref{exstarrys} for an example of both situations. When a $k$-star pattern $\mathcal{P}$ is non-simplicial (in particular, $k\ge3$),
the unique discrete component of cardinality $k$ will be called the \emph{branching component}. In this case, for any model $(T,P,f)$ of $\mathcal{P}$
such that $T$ is a $k$-star, the convex hull of the branching component of $\mathcal{P}$ contains the central point of $T$.

\begin{figure}
\centering
\includegraphics[scale=0.5]{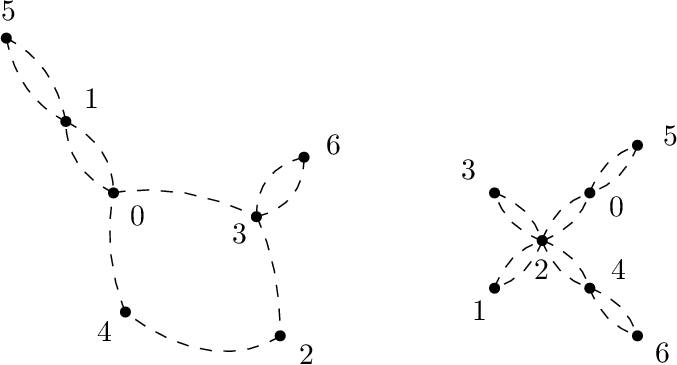}
\caption[fck]{Two examples of 4-star patterns of period 7, satisfying respectively (NS) (left) and (S) (right).\label{exstarrys}}
\end{figure}

To prove Theorem~\ref{mainC}, in particular its ``if'' part, we must elucidate for which values $n,k$ there is an example of a zero
entropy map $\map{f}{T}$, where $T$ is a $k$-star, having an $n$-periodic orbit $P$ with at least one point in each branch of $T$.
In this case, the pattern $\mathcal{P}:=[T,P,f]$ has entropy zero and is a star pattern. It is clear, then, that in order to find such examples
we need a  descriptive way to construct star patterns of entropy zero. Precisely, Theorem~\ref{mainB} does the job: if $h(\mathcal{P})=0$, then $\mathcal{P}$
it strongly collapsible and there exists a sequence of collapses of $\mathcal{P}$ according to Definition~\ref{explosions}.

Assume that we want to construct an example of a (general, not necessarily star) zero entropy pattern of period $n$. Decompose $n$ into a
product of integers $n=p_0p_1\cdots p_r$, with all $p_i$ larger than 1, and
start with a trivial pattern $\mathcal{P}_0$ of $p_0$ points. Now, to construct $\mathcal{P}_1$, ``explode'' each point $i\in\{0,1,\ldots,p_0-1\}$ of
$\mathcal{P}_0$ into a trivial block of $p_1$ points, $\{i,i+p_0,i+2p_0,\ldots,i+(p_1-1)p_0\}$, making sure that the structure defined by the new blocks
is maximal in order that $\mathcal{P}_0$ is the combinatorial collapse of $\mathcal{P}_1$. We will say that $\mathcal{P}_1$ is a \emph{$p_1$-explosion} of $\mathcal{P}_0$.
And so on. For an example with $n=18$, $p_0=3$, $p_1=2$, $p_2=3$, look at Figure~\ref{ullh0} (thinking that the arrows in the picture are now reversed).
In this spirit, if $\mathcal{P}$ is an $n$-periodic pattern such that $h(\mathcal{P})=0$, from now on the sequence $(p_0,p_1,\ldots,p_r)$ in
Definition~\ref{explosions} will be called \emph{the sequence of explosions of $\mathcal{P}$}. In this case, $n=p_0p_1\cdots p_r$.

In the particular case of star patterns, the above process has very particular restrictions. Let us see it.

\begin{lemma}\label{basic1}
Let $\mathcal{P}$ be a periodic star pattern such that $h(\mathcal{P})=0$. Let $\{\mathcal{P}_i\}_{i=0}^r$ be the sequence of collapses of $\mathcal{P}$
and let $(p_i)_{i=0}^r$ be the corresponding sequence of explosions. Then,
\begin{enumerate}
\item[(a)] Each $\mathcal{P}_i$ is a $k_i$-star pattern for some $k_i\ge2$.
\item[(b)] $k_i\le k_{i+1}$.
\item[(c)] If $\mathcal{P}_i$ is non-simplicial, then $\mathcal{P}_{i+1}$ is non-simplicial
\end{enumerate}
\end{lemma}
\begin{proof}
Note that if $\mathcal{P}_i$ is not a star pattern, i.e. their points cannot be embedded in a star, the same happens with $\mathcal{P}_{i+1}$,
where each point of $\mathcal{P}_i$ has been exploded to a block. Since $\mathcal{P}_r=\mathcal{P}$ is a star pattern, then (a) follows by backward induction.
To see (b), observe that when performing a combinatorial collapse, the number of endpoints decreases (or remains equal), and that $k_i$ is precisely
the number of endpoints of $\mathcal{P}_i$. Finally, to prove (c) recall that $\mathcal{P}_i$ is non-simplicial if and only if it has a branching component
$C$ of $k_i\ge3$ points. Since $\mathcal{P}_i$ is the combinatorial collapse of $\mathcal{P}_{i+1}$, by definition $\mathcal{P}_{i+1}$ has a discrete
component $C'$ intersecting $k_i$-many different blocks. So, $C'$ is a branching component of $\mathcal{P}_{i+1}$ with $|C'|\ge k_i\ge3$, and
$\mathcal{P}_{i+1}$ is non-simplicial.
\end{proof}

\begin{lemma}\label{binary}
Assume that a pattern with a structure of trivial blocks has a discrete component of two points $\{a,b\}$ such that $a$ is an endpoint. Then,
all blocks of the trivial structure have cardinality 2.
\end{lemma}
\begin{proof}
Since a block has at least two points and $a$ is an endpoint, it is obvious that $a,b$ belong to the same block, $K$. Since the blocks are trivial,
$K$ is contained in a single discrete component by definition. It follows that $K=\{a,b\}$. In consequence, all blocks of the trivial structure have
cardinality 2.
\end{proof}

\begin{proposition}\label{generalST}
Let $\mathcal{P}$ be an $n$-periodic star pattern such that $h(\mathcal{P})=0$. Let $\{\mathcal{P}_i\}_{i=0}^r$ be the sequence of collapses of $\mathcal{P}$
and let $(p_i)_{i=0}^r$ be the corresponding sequence of explosions. Then,
$p_1=p_2=\ldots=p_r=2$. Thus, $n=p_02^r$.
\end{proposition}
\begin{proof}
By Lemma~\ref{basic1}(a), $\mathcal{P}_i$ is a star pattern for $0\le i\le r$. Remark~\ref{captrivial} tells us that $\mathcal{P}_i$ is
nontrivial when $i\ne 0$. Now observe that any nontrivial star pattern has at least one binary discrete component $\{a,b\}$ such that $a$ is an endpoint.
From Lemma~\ref{binary}, it follows that $p_i=2$ for each $i\ne 0$.
\end{proof}

Let $\mathcal{P}$ be a zero entropy periodic star pattern. Let $\{\mathcal{P}_i\}_{i=0}^r$ be the sequence of collapses of $\mathcal{P}$.
By Lemma~\ref{basic1}(a), there exists a unique sequence of integers $(k_i)_{i=0}^r$ such that $\mathcal{P}_i$ is a $k_i$-star pattern. It
will be called the \emph{sequence of valences} of $\mathcal{P}$. The next remark collects several obvious items.

\begin{figure}
\centering
\includegraphics[scale=0.6]{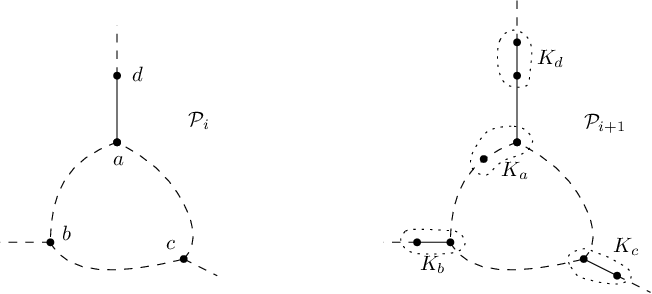}
\caption[fck]{Local picture of an expanding 2-explosion of type EE1.\label{EE1}}
\end{figure}

\begin{remark}\label{basic2}
Let $(k_i)_{i=0}^r$ be the sequence of valences associated to a sequence of collapses $\{\mathcal{P}_i\}_{i=0}^r$ of a star pattern. Then,
\begin{enumerate}
\item[(a)] $k_i$ coincides with the number of endpoints of $\mathcal{P}_i$.
\item[(b)] When $\mathcal{P}_i$ is non-simplicial, $k_i$ is also the cardinality of the branching component.
\item[(c)] When $\mathcal{P}_i$ is simplicial (and $i\ne 0$), $k_i$ is the maximum valence of a point of $\mathcal{P}_i$. If $k_i\ge3$, there is only
one point attaining the maximum valence $k_i$.
\end{enumerate}
\end{remark}

If $k_{i+1}>k_i$ (equivalently, $\mathcal{P}_{i+1}$ has more endpoints than $\mathcal{P}_i$), we will say
that $\mathcal{P}_{i+1}$ is an \emph{expanding explosion of $\mathcal{P}_i$}. Let us examine all the ways this can happen. First of all, by Proposition~\ref{generalST},
$\mathcal{P}_{i+1}$ is a 2-explosion of $\mathcal{P}_i$. So, each point $a$ in $\mathcal{P}_i$ is replaced by a block of two points $K_a$ in $\mathcal{P}_{i+1}$.
We have the following possibilities for an expanding 2-explosion.
\begin{enumerate}
\item[(EE1)] If $\mathcal{P}_i$ if non-simplicial, then $\mathcal{P}_{i+1}$ is also non-simplicial by Lemma~\ref{basic1}(c). For at least one point $a$ of the branching component of
$\mathcal{P}_i$, the branching component of $\mathcal{P}_{i+1}$ contains the two points of the block $K_a$, and one of the two points is an endpoint. See Figure~\ref{EE1}.
\item[(EE2)] If both $\mathcal{P}_i$ and $\mathcal{P}_{i+1}$ are simplicial, let $a$ be a point of $\mathcal{P}_i$ of valence $k_i$. In $\mathcal{P}_{i+1}$, the block $K_a$
is a full discrete component of two points. One of them is an endpoint and the other one has valence $k_i+1$. See Figure~\ref{EE2}. Note that this construction cannot be performed
when $\mathcal{P}_i=\mathcal{P}_0$. In this case, $\mathcal{P}_0$ is the trivial pattern of two points and, if $\mathcal{P}_1$ is simplicial, it is necessarily a 4-periodic
interval pattern, which is 2-star. Thus, $k_0=k_1=2$ and the explosion is non-expansive.
\begin{figure}
\centering
\includegraphics[scale=0.6]{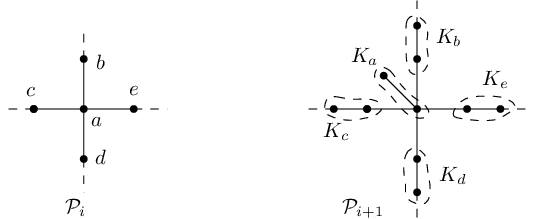}
\caption[fck]{Local picture of an expanding 2-explosion of type EE2.\label{EE2}}
\end{figure}
\item[(EE3)] Assume that $\mathcal{P}_i$ is simplicial and $\mathcal{P}_{i+1}$ is non-simplicial. This case is remarkable and very restrictive. Indeed, all possible
2-explosions (expansive or not) are drawn in Figure~\ref{EE3}. It turns out that $\mathcal{P}_{i+1}$ can be a 2-explosion (expansive or not) of $\mathcal{P}_i$
if and only if $k_i=2,3$. In other words, it is not possible to obtain a a non-simplicial pattern from an expanding 2-explosion of a simplicial pattern having a
point of valence bigger than 3. To see it, add a discrete component $\{a,e\}$ to the (local) distribution $\{a,b\},\{a,c\},\{a,d\}$ drawn in Figure~\ref{EE3} (bottom).
The block $K_e$ should share a discrete component with the block $K_a$, but \emph{not} with the blocks $K_b,K_c,K_d$. This would add an extra valence to one of
the two points of the block $K_a$, a contradiction since a star pattern cannot have points of valence 3.
\end{enumerate}

\begin{figure}
\centering
\includegraphics[scale=0.6]{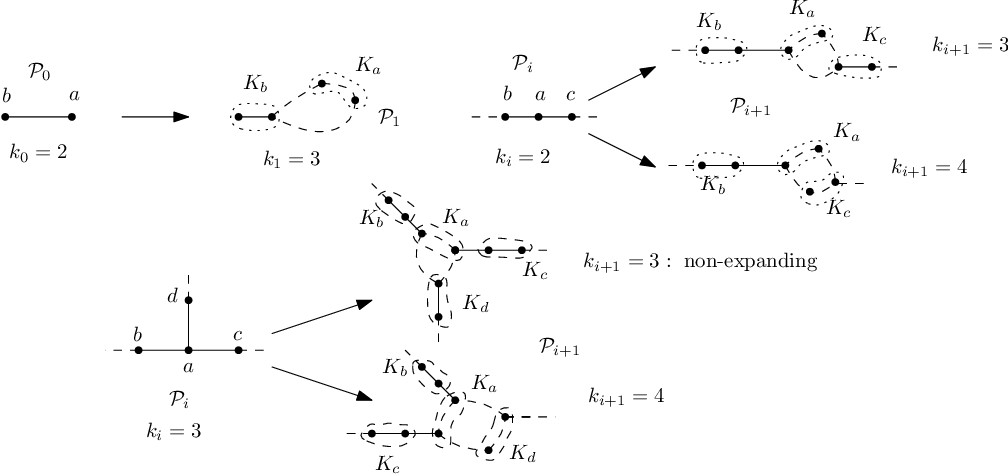}
\caption[fck]{Local pictures of all possible 2-explosions from a simplicial pattern to a non-simplicial one. All possibilities
are EE3 (expanding) except the case $k_i=k_{i+1}=3$. The case $\mathcal{P}_i=\mathcal{P}_0$
is special: the blocks $K_a,K_b$ cannot be fully contained in a single discrete component for otherwise $\mathcal{P}_1$ would
be a trivial pattern. \label{EE3}}
\end{figure}

The previous discussion leads to the following technical result.

\begin{lemma}\label{endpoint}
Let $\mathcal{P}$ and $\mathcal{Q}$ be zero entropy star periodic patterns such that $\mathcal{Q}$ is an expanding 2-explosion of $\mathcal{P}$.
If $\mathcal{Q}$ is non-simplicial, then there is at least one block of the maximal structure of $\mathcal{Q}$ intersecting the branching component. Moreover,
at least one of the two points of the block is an endpoint.
\end{lemma}
\begin{proof}
The explosion has to be either EE1 if $\mathcal{P}$ is non-simplicial or EE3 if $\mathcal{P}$ is simplicial. In each case the lemma holds, see Figures~\ref{EE1}
and \ref{EE3}.
\end{proof}

The above considerations have also a remarkable consequence. For an $n$-periodic simplicial $k$-star pattern of entropy zero, there is an upper bound for
the valence $k$ of the central point of the star in terms of the period $n$, as the following result states.

\begin{lemma}\label{basic3}
Let $\mathcal{P}$ be a nontrivial zero entropy $n$-periodic $k$-star pattern. Let $\{\mathcal{P}_i\}_{i=0}^r$ be the sequence of collapses of $\mathcal{P}$.
If $\mathcal{P}$ is simplicial, then $n=2^{r+1}$ and $k\le r+1$.
\end{lemma}
\begin{proof}
Let $(p_i)_{i=0}^r$ be the sequence of explosions of $\mathcal{P}$. Then, $n=p_0p_1\cdots p_r$. By Proposition~\ref{generalST}, $p_1=p_2=\ldots=p_r=2$.
Thus, $n=p_02^r$. Since $\mathcal{P}$ is simplicial, Lemma~\ref{basic1}(c) implies that all patterns $\mathcal{P}_i$ are simplicial. In particular, $\mathcal{P}_0$ is
the trivial pattern of two points. Hence, $p_0=2$ and $n=2^{r+1}$. Now let $(k_i)_{i=0}^r$ be the sequence of valences of $\mathcal{P}$. We have that $k_0=k_1=2$ and
$k_r=k$. The maximum possible value for $k$ is attained when $\mathcal{P}_{i+1}$ is an expanding 2-explosion of $\mathcal{P}_i$ for all $1\le i\le r-1$. In this case,
the 2-explosion is of type EE2 and $k_{i+1}=k_i+1$. Since $k_1=2$ and $k_r=k$, we get that $k\le r+1$.
\end{proof}

In particular, the previous lemma extends the well known result \cite{bgmy} that the period of any zero entropy interval pattern is a power of 2.
For general tree patterns, this result can also be derived immediately from Lemma~4.1 of \cite{forward}.

\section{Proof of Theorem~\ref{mainC}}\label{S7}
The next two results give the examples to prove the ``if'' part of Theorem~\ref{mainC}.

\begin{lemma}\label{1r}
Let $k\ge 3$ and $n=k\cdot 2^q$ for some $q\ge0$. Then, there exists a zero entropy map $\map{f}{T}$ such that $T$ is a $k$-star and $f$ has an $n$-periodic
orbit having at least one point in each branch of $T$.
\end{lemma}
\begin{proof}
Let $y$ be the central point of $T$. If $q=0$, then $n=k$. In this case, we define $f$ has a rigid rotation of $T$, with $f(y)=y$.
Its entropy is zero and the set of $k$ endpoints of $T$ is an $n$-periodic orbit.

Assume that $q>0$. Let $\mathcal{P}_0$ be the trivial pattern of $k$ points. Now, let $\mathcal{P}$ be the pattern of period $k\cdot 2^q$ obtained
from $\mathcal{P}_0$ by doing $q$ successive non-expanding 2-explosions. Since $h(\mathcal{P}_0)=0$, Proposition~\ref{SC} tells us that $h(\mathcal{P})=0$.
Of course this pattern is not unique, see Figure~\ref{1rfig} for two different examples with $k=3$, $q=1$. Now consider the model $(T,P,f)$ of $\mathcal{P}$
such that $f(y)=y$ and $f$ is monotone on any $(P\cup\{y\})$-basic interval. It is easy to see that $(T,P,f)$ is the canonical model of $\mathcal{P}$.
Then, $h(f)=h(\mathcal{P})=0$.
\end{proof}

\begin{lemma}\label{2n}
Let $k\ge 3$ and $n=2^q$ for some $q\ge k-1$. Then, there exists a zero entropy map $\map{f}{T}$ such that $T$ is a $k$-star and $f$ has an $n$-periodic
orbit with at least one point in each branch of $T$.
\end{lemma}
\begin{proof}
Assume first that $q\ge k$ (we will treat separately the case $q=k-1$). Let $\mathcal{P}_2$ be the 4-periodic zero
entropy (interval) pattern whose discrete components are $\{2,0\}$, $\{0,1\}$, $\{1,3\}$. In particular, the valence of the point 0 is 2.
Now, perform the following iterative procedure. For $3\le i\le k$, let $\mathcal{P}_i$ be an expanding 2-explosion of $\mathcal{P}_{i-1}$ of type EE2, in
such a way that:
\begin{enumerate}
\item[(1)] $\mathcal{P}_i$ is an $i$-star simplicial pattern.
\item[(2)] 0 is the unique point of $\mathcal{P}_i$ with valence $i$ and $2^{i-1}$ is an endpoint.
\end{enumerate}
See Figure~\ref{redeu} for an example of the previous iterative procedure with $k=4$. Note that $\mathcal{P}_i$ satisfies:
\begin{enumerate}
\item[(3)] The period of $\mathcal{P}_i$ is $2^i$.
\item[(4)] $h(\mathcal{P}_i)=0$.
\end{enumerate}

\begin{figure}
\centering
\includegraphics[scale=0.7]{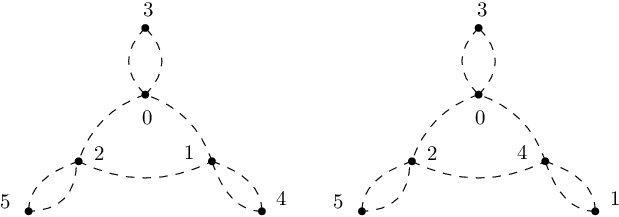}
\caption[fck]{Two examples of 6-periodic patterns considered in the proof of Lemma~\ref{1r}. \label{1rfig}}
\end{figure}

Property (4) follows from Proposition~\ref{SC}, since $\mathcal{P}_i$ has been obtained from $\mathcal{P}_2$ by performing 2-explosions and $h(\mathcal{P}_2)=0$.

\begin{figure}
\centering
\includegraphics[scale=0.6]{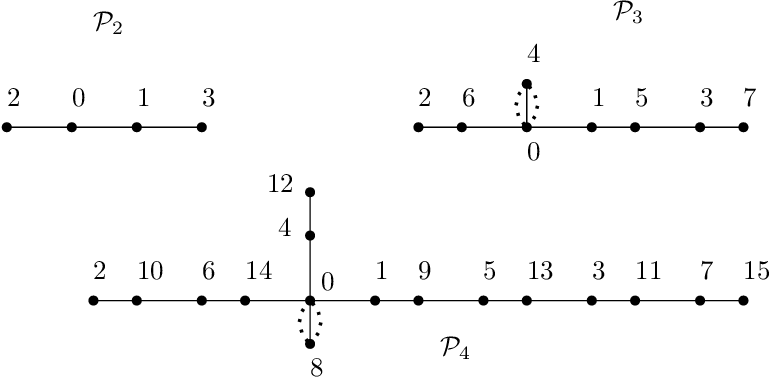}
\caption[fck]{Iterative procedure for constructing the pattern $\mathcal{P}_k$ in the proof of Lemma~\ref{2n}, here with $k=4$.\label{redeu}}
\end{figure}

Finally, set $\mathcal{Q}_k:=\mathcal{P}_k$ and, for $k+1\le i\le q$, let $\mathcal{Q}_i$ be a non-expanding 2-explosion of $\mathcal{Q}_{i-1}$. So,
$\mathcal{Q}_i$ remains a $k$-star simplicial pattern for all $i$. In this second part of the procedure, the exploding blocks do not create new
directions, in particular we do not modify the valence $k$ of the point 0.

Since $\mathcal{Q}_q$ has been obtained from $\mathcal{P}_k$ by performing 2-explosions and $h(\mathcal{P}_k)=0$, then $h(\mathcal{Q}_q)=0$ by
Proposition~\ref{SC}. Moreover, the period of $\mathcal{Q}_q$ is $2^q=n$. Let $(T,P,f)$ be the canonical model of $\mathcal{Q}_q$. Then,
$h(f)=0$. In addition, since $\mathcal{Q}_q$ is a $k$-star simplicial pattern, $T$ is a $k$-star. So, the map $\map{f}{T}$ fulfills the required
properties and we are done.

Finally we are left with the case $n=2^{k-1}$. Take the pattern $\mathcal{P}_{k-1}$ defined above, satisfying (1--4) with $i=k-1$.
In particular, $[0,2^{k-2}]$ is a basic interval, the valence of $0$ is $k-1$ and $2^{k-2}$ is an endpoint. Let $(S,Q,g)$ be the canonical model of
$\mathcal{P}_{k-1}$. Then, $h(g)=0$. Since $\mathcal{P}_{k-1}$ is a ($k$--1)-star simplicial pattern, $S$ is a ($k$--1)-star. Now consider the $k$-star
$T$ obtained from $S$ by replacing the interval $[0,2^{k-2}]$ with two intervals $[v,0]$ and $[v,2^{k-2}]$, in such a way that the valence of $v$ in $T$ is $k$,
and $0$ and $2^{k-2}$ are endpoints of $T$. So, $T$ is a $k$-star with $v$ as its central point. See Figure~\ref{redeu2}
for an example with $k=4$. Let $\map{f}{T}$ be the map such that $f(i)=g(i)$ for $0\le i<n$ and $f(v)=1$, and
$f$ is monotone on each $(Q\cup\{v\})$-basic interval. In particular, $f$ collapses the edge $[v,0]$ into the point 1. It is easy to see
that $h(f)=h(g)$ (see, for instance, the final step in the proof of Theorem~\ref{5.3} of \cite{ajm}).
Since $h(g)=0$, $\map{f}{T}$ is a zero entropy $k$-star map having a periodic orbit $Q$ of period $n$ with points in each branch of $T$.
\end{proof}

\begin{figure}
\centering
\includegraphics[scale=0.6]{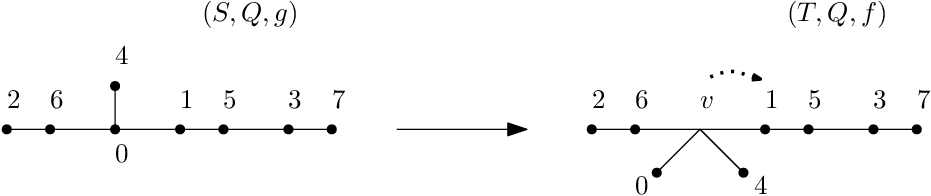}
\caption[fck]{Extending a 3-star map to a 4-star map with the same entropy as in the proof of Lemma~\ref{2n}.\label{redeu2}}
\end{figure}

We need a last ingredient to prove Theorem~\ref{mainC}. It is a well-known criterion of positive entropy for tree maps (see for instance
Lemma~6.1 of \cite{lm}).

\begin{lemma}\label{LMprevi}
Let $\map{g}{S}$ be a tree map. Let $I,J_1,J_2\subset S$ be closed intervals containing no points of $V(S)$ in their interiors such that $\Int(J_1)\cap
\Int(J_2)=\emptyset$. If there exist positive integers $\ell,\ell_1,\ell_2$ such that $f^\ell(I)\supset J_1\cup J_2$, $f^{\ell_1}(J_1)\supset I$ and
$f^{\ell_2}(J_2)\supset I$, then $h(g)>0$.
\end{lemma}

Now we are ready to prove Theorem~\ref{mainC}.

\begin{proof}[Proof of Theorem~\ref{mainC}]
If $n=k\cdot 2^q$ for some $q\ge0$, then $M_{n,k}=0$ by Lemma~\ref{1r}. On the other hand, if $n=2^q$ for some $q\ge k-1$ then we also have
$M_{n,k}=0$, in this case by Lemma~\ref{2n}. So, the ``if'' part of the theorem is proven.

Assume now that $M_{n,k}=0$. Then, there exists a zero entropy map $\map{f}{T}$, where $T$ is a $k$-star, having an $n$-periodic orbit $P$ with at least
one point in each branch of $T$. Let $y$ be the central point of $T$. By the Main Theorem of \cite{AMM}, we can assume that $f(y)\in P\cup\{y\}$ and that
$f$ is monotone on each $(P\cup\{y\})$-basic interval. In other words, we can take $(T,P\cup\{y\},f)$ to be a Markov model.

Take the pattern $\mathcal{P}:=[T,P,f]$. Since $h(f)=0$, then $h(\mathcal{P})=0$. By Theorem~\ref{mainB}, $\mathcal{P}$ is strongly
collapsible. Let $\{\mathcal{P}_i\}_{i=0}^r$ be the sequence of collapses of $\mathcal{P}$ and let $(p_i)_{i=0}^r$ be the corresponding
sequence of explosions. Since $T$ is a $k$-star, $\mathcal{P}$ is a $k$-star pattern. Then, by Proposition~\ref{generalST}, the sequence of explosions is
$(p_0,2,2,\stackrel{r)}{\ldots},2)$ for some $2\le p_0\le k$, and
\begin{equation}\label{peri}
n=p_0 2^r.
\end{equation}

Let $(k_i)_{i=0}^r$ be the sequence of valences of $\mathcal{P}$. Define $t$ as the minimum index in $\{0,1,\ldots,r\}$ such that
\begin{equation}\label{tes}
k_t=k_{t+1}=\cdots=k_r=k.
\end{equation}

If $t=0$, then (\ref{tes}) yields $k_0=k$. Since $k_0=p_0$, (\ref{peri}) reads now as $n=k\cdot 2^r$. Then, statement (a) is satisfied with $q:=r$ and the theorem is
proven in this case. So, from now on we assume that $t\ge1$.

If there exists some $i\in\{t,t+1,\ldots,r\}$ for which $\mathcal{P}_i$ is simplicial, then all patterns
\[ \mathcal{P}_0,\mathcal{P}_1,\ldots,\mathcal{P}_{i-1}\]
are also simplicial by Lemma~\ref{basic1}(c). In particular, $p_0=2$ and (\ref{peri}) reads as $n=2^{r+1}$. Moreover, applying Lemma~\ref{basic3}
to $\mathcal{P}_i$ yields $k_i\le i+1$. Summarizing, we have $k=k_i\le i+1\le r+1$ and $n=2^{r+1}$. Thus, statement (b) is satisfied with $q:=r+1$ and the theorem is
proven in this case. 

So, we are left with the possibility that $t\ge 1$ and 
\begin{equation}\label{totsS1}
\mathcal{P}_i\mbox{ is non-simplicial for all }t\le i\le r.
\end{equation}
Moreover, $k_t=k_{t+1}=\cdots=k_r$ and $k_{t-1}<k_t$.

Let $C$ be the branching component of $\mathcal{P}_r$ and denote the blocks of $\mathcal{P}_r$ by $K_j:=\{j,j+2^r\}$. If $k_r=k_{r-1}$ (in other words if $t<r$),
then $\chull{K_j}\cap\Int(\chull{C})=\emptyset$ for every block $K_j$. It follows that if $T'$ is the tree obtained from $T$ by collapsing each $\chull{K_j}$ to a
single point (labeled as $j$), then $T'$ is still a $k$-star with central point $y$. Define now a map $f'$ on $P':=\{0,1,\ldots,2^r-1\}$ as $f'(j)=j+1\bmod 2^r$.
Set also $f'(y)=y$ if $f(y)=y$, or $f'(y)=j$ if $f(y)\in K_j$. Finally, extend $f'$ from $P'\cup\{y\}$ to the whole tree $T'$ by taking an obvious piecewise extension,
in such a way that $(T',P'\cup\{y\},f')$ is a Markov model. It is easy to see that the entropies of $f$ and $f'$ are the same (the proof of this fact is similar to that
of Proposition~\ref{81}). So, $h(f')=0$. Moreover, the pattern $[T',P',f']$ is precisely $\mathcal{P}_{r-1}$.

By iteration ($r-t$ times) of the previous construction, we get a Markov model $(S,Q\cup\{y\},g)$ such that $[S,Q,g]=\mathcal{P}_t$, $S$ is a
$k$-star with central point $y$, and $h(g)=0$. Note that $Q$ is a $2^{t+1}$-periodic orbit having a maximal structure $Q_0\cup Q_1\cup\cdots\cup Q_{2^t-1}$
of blocks of cardinality 2. Let $D$ be the branching component of $\mathcal{P}_t$. Now, we distinguish two cases.

\begin{case}{1} There exists $j\in Q\cap D$ such that $g(y)=g(j)$. \end{case}
In this case, since $g$ is $(Q\cup\{y\})$-monotone, $g$ collapses the basic
interval $[y,j]$ into a single point, $j+1$. Consider the tree $S'$ obtained after replacing the interval $[y,j]$ by a single point $j$. Note that the
valence $k'$ of $j$ in the new tree $S'$ is either $k-1$ if $j$ was an endpoint of the tree $S$, or $k$ otherwise:
\begin{equation}\label{valj}
k'\in\{k,k-1\}\mbox{, where $k'$ is the valence of $j$ in $S'$}.
\end{equation}
Moreover, all discrete components of the pointed tree $(S',Q)$ have 2 points. Let $\map{g'}{S'}$ be the map such that $g'(i)=g(i)$ for $0\le i<2^t$ and $g'$ is monotone
on each $Q$-basic interval. As it was mentioned in the proof of Lemma~\ref{2n}, the fact that $g$ maps identically the interval $[y,j]$ into the point $j+1$ easily implies
that $h(g')=h(g)=0$. So, $\mathcal{Q}':=[S',Q,g']$ is a zero entropy simplicial pattern with period $|Q|=2^{t+1}=n/2^{r-t}$, which, using (\ref{peri}), is equal to $p_02^t$.
Hence, $p_0=2$ and (\ref{peri}) reads as $n=2^{r+1}$. Since the sequence of explosions of $\mathcal{Q}'$ is $(2,\stackrel{t+1)}{\ldots},2)$, Lemma~\ref{basic3} tells us
that $k'\le t+1$. So, $k'\le r+1$ because $t\le r$ by definition. From (\ref{valj}), statement (b) is satisfied with $q:=r+1$, and the theorem is proven in this case.

\begin{case}{2} Either $g(y)=y$ or $g(y)\in Q\setminus g(D)$. \end{case}
Let us see that this assumption forces positive entropy for $g$, a contradiction telling us that this case does not apply. We start with a simple observation:
since $\bigcup_j Q_j$ is a $2^t$-block structure, the image by $g$ of the convex hull of a block $\chull{i_1,i_2}$ is a connected set that contains $g(i_1)$ and $g(i_2)$,
so it contains $\chull{g(i_1),g(i_2)}$. It follows that
\begin{equation}\label{bloki}
\mbox{for any block $Q_j$ and any $m\ge 0$, $g^m(\chull{Q_j})\supset\chull{Q_{j+m\bmod 2^t}}$.}
\end{equation}

Now we claim that
\begin{equation}\label{power}
\mbox{for any basic interval $J$ and any block $Q_j$, $g^\ell(J)\supset\chull{Q_j}$ for some $\ell\ge 0$.}
\end{equation}

Let us prove this claim. When $J$ is the convex hull of a block, the claim follows trivially from (\ref{bloki}). When $J=[i_1,i_2]$ with
$\{i_1,i_2\}\subset Q$ but $i_1,i_2$ belong to different blocks, then $g^m(i_1)$ and $g^m(i_2)$ belong to different blocks for all $m$. Since
the pattern $[S,Q,g]$ is not trivial, there exists at least one block $K=\{i_3,i_4\}$ such that $i_3\in\En(S)\setminus D$. Take $m$ such that $g^m(i_1)=i_3$.
Then, $g^m(J)$ is a connected set containing $i_3$ and $g^m(i_2)\ne i_4$. The fact that $i_3$ is an endpoint implies $g^m(J)\supset[i_3,i_4]=\chull{K}$, and
the claim follows then from (\ref{bloki}). Finally, if $J=[y,i]$ with $i\in D$, then $g(J)\supset [g(y),g(i)]$, which is a proper interval (not necessarily
basic) because $g(y)\ne g(i)$ by the hypothesis of Case 2. Taking $m$ such that $g^m(g(y))=i_3$, the previous argument applies also here and the claim
follows.

By the choice of $t$ above, $\mathcal{P}_t$ is an expanding 2-explosion of $\mathcal{P}_{t-1}$. In addition, (\ref{totsS1}) says that $\mathcal{P}_t$ is
non-simplicial. Then, by Lemma~\ref{endpoint}, there is at least one block $Q_j$ such that $Q_j\subset D$ and at least one of the two points of $Q_j$ is an endpoint.
Shifting labels if necessary, we can assume without loss of generality that $Q_j=Q_0=\{0,2^t\}$ and that $0\in\En(S)$. We have
\begin{equation}\label{potser}
\chull{Q_0}=[0,y]\cup[y,2^t].
\end{equation}

Note that not all blocks can be contained in $D$, for otherwise $\mathcal{P}_t$ would be a trivial pattern. So, there exists $1\le i<2^t$ such that
$\chull{Q_i}=[i,2^t+i]$ is a basic interval outside $\chull{D}$, i. e. satisfying $[i,2^t+i]\subset S\setminus\Int(\chull{D})$. See Figure~\ref{cas12}.

\begin{figure}
\centering
\includegraphics[scale=0.6]{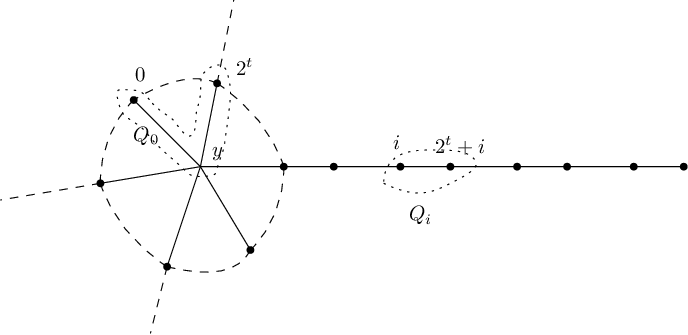}
\caption[fck]{Case 2 in the proof of Theorem~\ref{mainC}.\label{cas12}}
\end{figure}

Using (\ref{bloki}) with $j=i$ and $m=2^t-i$, together with (\ref{potser}), yields
\begin{equation}\label{universal1}
g^{2^t-i}([i,2^t+i])\supset [0,y]\cup[y,2^t].
\end{equation}
Now, (\ref{power}) tells us that there exist $\ell_1,\ell_2$ such that
\begin{equation}\label{universal2}
g^{\ell_1}([0,y])\supset[i,2^t+i]\mbox{ and }g^{\ell_2}([y,2^t])\supset[i,2^t+i].
\end{equation}

Now, collecting (\ref{universal1}) and (\ref{universal2}), the hypotheses of Lemma~\ref{LMprevi} are satisfied with
$I:=[i,2^t+i]$, $J_1:=[0,y]$, $J_2:=[y,2^t]$ and $\ell:=2^t-i$. Thus, $h(g)>0$.
\end{proof}

\end{document}